\documentclass[10pt]{article}

\usepackage{amsmath}
\usepackage{amsfonts}
\usepackage{amsthm}
\usepackage[english]{babel}
\usepackage{graphicx}
\usepackage[all]{xy}
\usepackage{amssymb}
\usepackage{pdflscape}
\usepackage{accents}
\usepackage{filecontents}

\newtheorem{theorem}{Theorem}[section]
\newtheorem{lemma}{Lemma}[section]
\newtheorem{definition}{Definition}[section]
\newtheorem{corollary}{Corollary}[section]

\newtheorem{remark}{Remark}[section]

\newcommand{\la}{\langle}
\newcommand{\ra}{\rangle} 

\newcommand{\mR}{\mathbb{R}}
\newcommand{\mC}{\mathbb{C}}
\newcommand{\mN}{\mathbb{N}}
\newcommand{\mE}{\mathbb{E}}

\newcommand{\mS}{\mathbb{S}}

\newcommand{\cM}{\mathcal{M}}

\newcommand{\cP}{\mathcal{P}}

\newcommand{\cH}{\mathcal{H}}

\newcommand{\ux}{\underaccent{\bar}{x}}

\newcommand{\uy}{\underaccent{\bar}{y}}

\newcommand{\uz}{\underaccent{\bar}{z}}
\newcommand{\uzd}{\underaccent{\bar}{z}^\dagger}
\newcommand{\uu}{\underaccent{\bar}{u}}
\newcommand{\uud}{\underaccent{\bar}{u}^\dagger}

\newcommand{\zc}{\bar{z}}
\newcommand{\uc}{\bar{u}}
\newcommand{\dbar}{\bar{\partial}}

\newcommand{\upx}{\partial_{\underaccent{\bar}{x}}}
\newcommand{\upy}{\partial_{\underaccent{\bar}{y}}}
\newcommand{\upz}{\partial_{\underaccent{\bar}{z}}}
\newcommand{\upu}{\partial_{\underaccent{\bar}{u}}}
\newcommand{\upzd}{\partial_{\underaccent{\bar}{z}}^\dagger}
\newcommand{\upud}{\partial_{\underaccent{\bar}{u}}^\dagger}

\def\l{{\lambda}}

\begin{document}
\title{Reproducing kernels for polynomial null-solutions of Dirac operators}

\author{H.\ De Bie\footnote{E-mail: {\tt Hendrik.DeBie@UGent.be}} \and F. Sommen\footnote{E-mail: {\tt Frank.Sommen@UGent.be}} \and M. Wutzig\footnote{E-mail: {\tt Michael.Wutzig@UGent.be}}}

\vspace{10mm}

\date{
\small{Department of Mathematical Analysis\\ Faculty of Engineering and Architecture - Ghent University}\\
\small{Galglaan 2, 9000 Gent, Belgium}\\
}



\maketitle

\begin{abstract}
It is well-known that the reproducing kernel of the space of spherical harmonics of fixed homogeneity is given by a Gegenbauer polynomial. By going over to complex variables and restricting to suitable bihomogeneous subspaces, one obtains a reproducing kernel expressed as a Jacobi polynomial, which leads to Koornwinder's celebrated result on the addition formula.

In the present paper, the space of Hermitian monogenics, which is the space of polynomial bihomogeneous null-solutions of a set of two complex conjugated Dirac operators, is considered.
The reproducing kernel for this space is obtained and expressed in terms of sums of Jacobi polynomials. This is achieved through use of the underlying Lie superalgebra $\mathfrak{sl}(1|2)$, combined with the equivalence between the $L^2$ inner product on the unit sphere and the Fischer inner product. The latter also leads to a new proof in the standard Dirac case related to the Lie superalgebra  $\mathfrak{osp}(1|2)$.
\end{abstract}

\section{Introduction}
Various dual pairs in the sense of Howe can be obtained in explicit realizations through the study of sets of differential operators of Laplace or Dirac type. In this paper we are concerned with four such dual pairs:
\begin{itemize}
\item the real Laplace operator governed by $\mathfrak{sl}(2) \times O(m)$, and its Dirac refinement $\mathfrak{osp}(1|2)\times Spin(m)$,  
\item the complex Laplace operator governed by $\mathfrak{gl}(2) \times U(n)$, and its Dirac refinement $\mathfrak{sl}(1|2)\times U(n)$.
\end{itemize}
In all these cases, a crucial role is played by spaces of polynomial null-solutions of the relevant differential operators. For $\mathfrak{sl}(2) \times O(m)$ they constitute e.g. the well-known  spaces of spherical harmonics.

The four situations mentioned above are intimately connected with the theory of orthogonal polynomials, which appear in the explicit expressions for the reproducing kernels of the related spaces of polynomial null-solutions. In the case of spherical harmonics, they are expressed as a single Gegenbauer polynomial \cite{AAR}. When going over to complex variables and considering bihomogeneous complex harmonics related to $\mathfrak{gl}(2) \times U(n)$ the resulting Jacobi polynomial follows from Koornwinder's celebrated addition formula \cite{koorn}.

While those two cases settle the situation for the Laplace operator, everything becomes quite a bit more complicated when considering Dirac operators instead. For the standard Euclidean Dirac operator, governed by the dual pair $\mathfrak{osp}(1|2)\times Spin(m)$, the reproducing kernel of the space of so-called spherical monogenics, i.e. homogeneous null-solutions of the Dirac operator, is known as a suitable sum of two Gegenbauers, see e.g. \cite{green} or \cite{deform}.
However, when performing a complex splitting of the Dirac operator, new interesting spaces of Hermitian monogenics, i.e. bihomogeneous null-solutions in the kernel of both Dirac operators, appear. It is the aim of the present paper to determine the reproducing kernels for these spaces, and to express them in terms of Jacobi polynomials as in the Koornwinder case. 

This will be achieved in two steps. First, we will develop a new way to determine the reproducing kernel in the case of the Euclidean Dirac operator (see Theorem 
\ref{action2} and Theorem \ref{RepKernel1}). Subsequently this method is adapted to the complex setting, thus leading to our two main results Theorem \ref{Dirac4} and Theorem \ref{RepKernel2}.

The main technical tools that we use to obtain these results are the (anti)commutation relations provided by the Lie superalgebra generated by the relevant differential operators, the equivalence between the $L^2$ inner product and the Fischer inner product \cite{Fi, S}, as well as intricate use of various Jacobi polynomial identities.

For a summary of our main results and a comparison of the four different cases, we refer the reader to Table \ref{table}.


The paper is organized as follows. In Section \ref{sec:1} we summarize the main results on polynomial null-solutions of Laplace type operators and state the main properties of Jacobi and Gegenbauer polynomials we will use in the rest of the paper. In Section \ref{sec:2} we treat the case of the Euclidean Dirac operator, to develop our new method. Finally, in Section \ref{sec:3} we treat the case of the two complex conjugated Dirac operators.

\section{Preliminaries}
\label{sec:1}
	\subsection{Spherical harmonics}
		We consider functions defined on $\mR^m$ with values in $\mC$. On $\mR^m$ we have the standard Euclidean inner product $\la x,y\ra=\sum_{j=1}^mx_jy_j$. 
		The norm of a vector is denoted by $|x|=\sqrt{\la x,x\ra}$. Using a multi-index notation, a monomial in $x=(x_1,\cdots,x_m)$ 
		can be written as $x^A=x_1^{\alpha_1}x_2^{\alpha_2}\cdots x_m^{\alpha_m}$ for an index vector $A=[\alpha_1,\cdots,\alpha_m]\in\mN^m$ with length 
		$|A|=\alpha_1+\cdots+\alpha_m$. The space of polynomials of degree $k$ is then
		\begin{align*}
			\Pi_k=\Big\{Q(x):Q(x)=\sum_{|\alpha|\leq k}c_\alpha x^\alpha\Big\},
		\end{align*}
		where the sum runs over all possible index vectors $\alpha$ with length at most $k$ and the coefficients $c_\alpha$ are complex numbers. On $\Pi_k$ we can 
		define an	inner product
		\begin{align*}
			\la P,Q\ra_\partial=[\overline{P(\partial)}Q(x)]_{x=0},
		\end{align*}
		the so-called Fischer inner product, see \cite{Fi,xu,S}. Here $\overline{P(\partial)}$ is the complex conjugate of the operator obtained by substituting the 
		partial derivative $\partial_{x_j}$ for	every variable $x_j$ in $P(x)$. Note that in the literature the complex conjugation is sometimes applied to the 
		second argument of this inner product. Here we have chosen this definition to keep it coherent with sections \ref{sec:2} and \ref{sec:3}.\\
		\indent A polynomial $P(x)$ is called homogeneous of degree $k$ if it holds for every $x\in\mR^m$ that $P(x)=|x|^kP(\frac{x}{|x|})$, where $\frac{x}{|x|}\in\mS^{m-1}$ is 
		the restriction of $x$ to the $(m-1)$-dimensional unit sphere $\mS^{m-1}=\{x\in\mR^m:|x|=1\}$. 
		The homogeneous polynomials of degree $k$ are the eigenfunctions of the Euler operator	$\mE=\sum_{j=1}^mx_j\partial_{x_j}$ with eigenvalue $k$. Therefore, 
		the space of homogeneous polynomials of degree $k$ is defined as
		\begin{align*}
			\cP_k=\big\{P(x):\mE P(x)=kP(x)\big\}.
		\end{align*}
		It can easily be verified that $\la\cdot,\cdot\ra_\partial$ is indeed an inner product as shown in \cite{xu}. The orthogonality of the spaces $\cP_k$ with respect to 
		$\la\cdot,\cdot\ra_\partial$, that is $\la P,Q\ra_\partial=0$	with $P(x)\in\cP_k$ and $Q(x)\in\cP_l$ for $k\neq l$, follows immediately.\\
		For each of the spaces $\cP_k$ there exists a unique reproducing kernel with respect to the Fischer inner product. This kernel is given by 
		$Z_k^m(x,y)=\frac{\la x,y\ra^k}{k!}$ and satisfies
		\begin{align*}
			\la Z_k^m(\cdot,y),P_l(\cdot)\ra_\partial=\delta_{kl}P_l(y),
		\end{align*}
		for any $P_l\in\cP_l$. On the space of polynomials $\Pi_k$ these kernels therefore work as projections onto the homogeneous spaces $\cP_j$, viz. for a polynomial
		$Q(x)=\sum_{j=0}^kP_j(x)$ of degree $k$ with $P_j\in\cP_j$ we have that
		\begin{align*}
			\la Z_j^m(\cdot,y),Q(\cdot)\ra_\partial=P_j(y).
		\end{align*}
		The space of $k$-homogeneous polynomials that are also harmonic, that is null-solutions of the Laplace operator $\Delta=\sum_{j=1}^m\partial_{x_j}^2$, is denoted by
		\begin{align*}
			\cH_k=\big\{P(x)\in\cP_k:\Delta P(x)=0\big\}.
		\end{align*}
		The restriction of such a homogeneous, harmonic polynomial to the unit sphere $\mS^{m-1}$ is called a spherical harmonic. However, because every homogeneous polynomial is
		uniquely defined by its values on the sphere we will also call $\cH_k$ the space of spherical harmonics of degree $k$.\\
		\indent As a result of the fact that every $k$-homogeneous polynomial $P_k\in\cP_k$ can be decomposed as $P_k(x)=H_k(x)+|x|^2Q_{k-2}(x)$, with $H_k\in\cH_k$ being 
		a spherical	harmonic and $Q_{k-2}\in\cP_{k-2}$ a homogeneous polynomial of lower degree (see also \cite{xu,Fi,S}), we have the Fischer decomposition
		\begin{align}
			\label{Fischer}
			\cP_k=\bigoplus_{j=0}^{\lfloor\frac{k}{2}\rfloor}|x|^{2j}\cH_{k-2j}.
		\end{align}
		This result states that every polynomial can be described in terms of spherical harmonics, which is the higher dimensional analogue of the	theory of Fourier series in two
		dimensions. As in the case $m=2$ spherical harmonics of different degrees are also orthogonal with respect to another inner product on $\Pi_k$, the	$L^2$ inner product 
		on the sphere
		\begin{align*}
			\la P,Q\ra_{\mS^{m-1}}=\frac{1}{\omega_{m-1}}\int_{\mS^{m-1}}\overline{P(x)}Q(x)d\sigma(x),
		\end{align*}
		where $d\sigma(x)$ is the spherical measure and $\omega_{m-1}=2\pi^\frac{m}{2}/\Gamma(\frac{m}{2})$ the surface of the $(m-1)$ dimensional
		unit sphere. Note that as for the Fischer inner product we apply the complex conjugation on the first argument. It was shown in \cite{xu} that these two 
		inner products are proportional for a spherical harmonic and a homogeneous polynomial of the same degree. The precise statement is as follows.
		\begin{theorem}
			For a spherical harmonic $H_k\in\cH_k$ and a homogeneous polynomial $P_k\in\cP_k$ it holds that
			\begin{align*}
				2^k\Big(\frac{m}{2}\Big)_k\la H_k,P_k\ra_{\mS^{m-1}}=\la H_k,P_k\ra_\partial.
			\end{align*}
			Moreover for two spherical harmonics of arbitrary order $H_k\in\cH_k$ and $Q_l\in\cH_l$, with $k\neq l$, one has that
			\begin{align*}
				2^k\Big(\frac{m}{2}\Big)_k\la H_k,Q_l\ra_{\mS^{m-1}}=\la H_k,Q_l\ra_\partial=0.
			\end{align*}
			Here $\big(\frac{m}{2}\big)_k=\frac{m}{2}(\frac{m}{2}+1)\cdots(\frac{m}{2}+k-1)$ denotes the Pochhammer symbol.
		\end{theorem}
		As for the space of homogeneous polynomials $\cP_k$ there also exits a unique reproducing kernel on the space of spherical harmonics $\cH_k$. In \cite{stein} 
		this kernel was called the zonal harmonic of degree $k$.
		\begin{theorem}
			\label{harmkernel1}
			For any $H_l\in\cH_l$ it holds that
			\begin{align*}
				\la K_k^m(\cdot,y),H_l(\cdot)\ra_{\mS^{m-1}}=\delta_{kl}H_l(y),
			\end{align*}
			with the reproducing kernel
			\begin{align*}
				K_k^m(x,y)=\frac{k+\mu}{\mu}|x|^k|y|^kC_k^\mu(t),
			\end{align*}
			where $\mu=\frac{m}{2}-1$, $t=\frac{\la x,y\ra}{|x||y|}$ and $C_k^\mu(t)$ the Gegenbauer polynomial of degree $k$ and order $\mu$.
		\end{theorem}
		Note that the spherical variable $t=\frac{\la x,y\ra}{|x||y|}=\cos(\phi)$ is just the cosine of the angle between the vectors $x$ and $y$. On the
		sphere the kernel $K_k^m(\frac{x}{|x|},\frac{y}{|y|})=\frac{\mu+k}{\mu}C_k^\mu(t)$ only depends on this spherical variable, hence the term \itshape zonal \upshape
		in \cite{stein}.
	\subsection{Complex harmonic analysis}
		If the vector space $\mR^m$ is of even dimension, i.e. $m=2n$, a complex structure can be applied to identify $\mR^{2n}$ with the complex vector space $\mC^n$. A complex
		vector $z=(z_1,\cdots,z_n)\in\mC^n$ then has complex coordinates $z_j=x_j+ix_{n+j}$. The complex conjugation is denoted by $\zc_j=x_j-ix_{n+j}$ and
		the Hermitian inner product between two complex vectors is $\la z,u\ra=\sum_{j=1}^nz_j\uc_j$. The partial differential operators are given by 
		\begin{align*}
			\partial_{z_j}&=\frac{1}{2}(\partial_{x_j}-i\partial_{x_{n+j}}),\\
			\dbar_{z_j}&=\frac{1}{2}(\partial_{x_j}+i\partial_{x_{n+j}}).
		\end{align*}
		Note that we denote a function that depends on the complex vectors $z$ and $\zc$, and hence on $x$, by $f(z)$ rather than $f(z,\zc)$ to simplify notations. 
		A function $f(z)$ will thus not in general be	holomorphic.\\
		It is now possible to split the Euler operator $\mE$ into two complex Euler operators
		\begin{align*}
			\mE_z&=\sum_{j=1}^nz_j\partial_{z_j},\\
			\mE_{\zc}&=\sum_{j=1}^n\zc_j\dbar_{z_j}.
		\end{align*}
		It is easily verified that $\mE_z+\mE_{\zc}=\mE$ allowing us to refine the notion of a $k$-homogeneous polynomial to a (bi-)homogeneous polynomial of degree $(p,q)$
		with $p+q=k$. The space of $(p,q)$-homogeneous polynomials is defined as
		\begin{align*}
			\cP_{p,q}=\{P(z)|\mE_z P(z)=pP(z),\mE_{\zc} P(z)=qP(z)\}.
		\end{align*}
		The space of $k$-homogeneous polynomials can then be understood as a direct sum of these $(p,q)$-homogeneous spaces, hence
		\begin{align*}
			\cP_k=\bigoplus_{j=0}^k\cP_{j,k-j}.
		\end{align*}
		The Fischer inner product defined on $\Pi_k$ can now be written as
		\begin{align*}
			\la P,Q\ra_\partial=[\overline{P(\partial)}Q(z)]_{z=0},
		\end{align*}
		where the operator polynomial $P(\partial)$ is obtained by substituting $2\dbar_{z_j}=\partial_{x_j}+i\partial_{x_{n+j}}$ for every $z_j$	and 
		$2\partial_{z_j}=\partial_{x_j}-i\partial_{x_{n+j}}$ for every $\zc_j$ in $P(z)$. As on $\cP_k$ there is also a reproducing kernel on the 
		space $\cP_{p,q}$ with respect to	$\la\cdot,\cdot\ra_\partial$. This kernel is given by $Z_{p,q}^n(z,u)=\frac{\la z,u\ra^p\overline{\la z,u\ra}^q}{p!q!}$ and satisfies
		\begin{align*}
			\la Z_{p,q}^n(\cdot,u),P_{s,t}(\cdot)\ra_\partial=\delta_{ps}\delta_{qt}P_{s,t}(u),
		\end{align*}
		for any bihomogeneous polynomial $P_{s,t}\in\cP_{s,t}$. On the space of polynomials of degree $k$ these kernels act as projections onto its subspaces $\cP_{s,t}$. 
		For a polynomial $Q(z)=\sum_{s=0}^p\sum_{t=0}^qP_{s,t}(z)$ of degree $k=p+q$ with $P_{s,t}\in\cP_{s,t}$ we have that
		\begin{align*}
			\la Z_{s,t}^n(\cdot,u),Q(\cdot)\ra_\partial=P_{s,t}(u).
		\end{align*}
		The	Laplace operator can be written in terms of the complex partial derivatives as $\Delta=4\sum_{j=1}^n\partial_{z_j}\dbar_{z_j}$. The space of spherical
		harmonics of order $(p,q)$ is then defined as the space of $(p,q)$-homogeneous polynomials that are in the kernel of $\Delta$, hence
		\begin{align*}
			\cH_{p,q}=\{H(z)|H(z)\in\cP_{p,q},\Delta H(z)=0\}.
		\end{align*}
		The reproducing kernel on $\cH_{p,q}$ with respect to the spherical inner product 
		\begin{align*}
			\la f,g\ra_{\mS^{2n-1}}=\frac{1}{\omega_{2n-1}}\int_{\mS^{2n-1}}\overline{f(z)}g(z)d\sigma(z)
		\end{align*}
		was derived by Koornwinder in \cite{koorn} in the context of establishing the addition formula for
		Jacobi polynomials. We summarize his result in he following theorem. The reader may also consult \cite{stras} for a more recent treatment.
		\begin{theorem}
			\label{harmkernel2}
			For any $H_{s,t}\in\cH_{s,t}$ it holds that
			\begin{align*}
				\la K_{p,q}^n(\cdot,u),H_{s,t}(\cdot)\ra_{\mS^{2n-1}}=\delta_{ps}\delta_{qt}H_{s,t}(u)
			\end{align*}
			with the reproducing kernel given by
			\begin{align*}
				K_{p,q}^n(z,u)=c_{p,q}\la z,u\ra^{p-q}\la z,z\ra^q\la u,u\ra^qP_q^{\nu,p-q}(2s-1),
			\end{align*}
			where $p\geq q$, $c_{p,q}=\frac{\nu+1+p+q}{\nu+1}$, $\nu=n-2$, the angular variable $s=\frac{\la z,u\ra\overline{\la z,u\ra}}{\la z,z\ra\la u,u\ra}$ and 
			$P_q^{\nu,p-q}(2s-1)$	the Jacobi polynomial of degree $q$ and parameters $\nu$ and $p-q$.
		\end{theorem}
	\subsection{Jacobi and Gegenbauer polynomials}
		Jacobi polynomials are a class of polynomials that are orthogonal with respect to the weight function $(1-x)^a(1+x)^b$ on the interval $[-1,1]$. The Jacobi
		polynomial $P_k^{a,b}(x)$ of degree $k$ with parameters $a,b\in\mR$, $a,b>-1$ is given by 
		\begin{align*}
			P_k^{a,b}(x)=\frac{\Gamma(a+k+1)}{k!\Gamma(a+b+k+1)}\sum_{j=0}^k\binom{k}{j}\frac{\Gamma(a+b+k+j+1)}{\Gamma(a+j+1)}\Big(\frac{x-1}{2}\Big)^j.
		\end{align*}
		As orthogonal polynomials they satisfy certain contiguous relations. The classic recurrence relations of the Jacobi polynomials $P_{k}^{a,b}(x)$ and the relation 
		with their derivatives can be found in \cite{askey,gabor} and are given by
		\begin{lemma}
			\label{JacRec1}
			For $a,b\in\mR$, $a,b>-1$ and $k\in\mN$ one has
			\begin{align}
				\label{jacobi1}P_{k+1}^{a+1,b}(x)-P_{k+1}^{a,b+1}(x)&=P_{k}^{a+1,b+1}(x)\\
				\label{jacobi2}(1-x)P_{k}^{a+1,b}(x)+(1+x)P_{k}^{a,b+1}(x)&=2P_{k}^{a,b}(x)\\
				\label{jacobi3}(k+a+b+2)P_{k+1}^{a,b+1}(x)+(k+a+1)P_{k}^{a,b+1}(x)&=(2k+a+b+3)P_{k+1}^{a,b}(x)\\
				\label{jacobi4}P_k^{a+1,b}(x)=\frac{\Gamma(k+b+1)}{\Gamma(k+a+b+2)}\sum_{j=0}^k(2j+a+b+1)&\frac{\Gamma(j+a+b+1)}{\Gamma(j+b+1)}P_j^{a,b}(x).
			\end{align}
			The derivative of a Jacobi polynomial is again a Jacobi polynomial, namely
			\begin{align}
				\label{jacobi5}\Big(P_{k}^{a,b}(x)\Big)'(x)=\frac{1}{2}(k+a+b+1)P_{k-1}^{a+1,b+1}(x).
			\end{align}
		\end{lemma}
		If one considers special cases of Jacobi polynomials more relations can be found. Important for us is the following case and the resulting recurrence relations.
		\begin{lemma}
			\label{JacRec2}
			For $p,q,n\in\mN$, $n\geq 2$ and $p>q$ it holds that
			\begin{align}
				\label{jacobi6}&(p-q)P_{q+1}^{n-2,p-q}(x)+2s\Big(P_{q+1}^{n-2,p-q}(x)\Big)'=(p+1)P_{q+1}^{n-1,p-q-1}(x)\\
				\label{jacobi7}&(q+1)P_{q+1}^{n-2,p-q}(x)-2s\Big(P_{q+1}^{n-2,p-q}(x)\Big)'=-(p+1)P_{q}^{n-1,p-q}(x)\\
				\label{jacobi8}&\binom{n-1+p}{p}P_q^{n-1,p-q}(x)=\sum_{j=0}^q\kappa_{p-j,q-j}P_{q-j}^{n-2,p-q}(x),
			\end{align}
			with $\kappa_{p-j,q-j}=\frac{n-1+p+q-2j}{n-1}\binom{n-2+p-j}{p-j}$ and $x=2s-1$.
		\end{lemma}
		\begin{proof}
			To prove equation (\ref{jacobi6}) one uses the derivative relation (\ref{jacobi5}) and $2s=x+1$, hence
			\begin{align*}
				(p-q)P_{q+1}^{n-2,p-q}(x)+2s\Big(P_{q+1}^{n-2,p-q}(x)\Big)'&=(p-q)P_{q+1}^{n-2,p-q}(x)+\frac{1}{2}(x+1)(n+p)P_q^{n-1,p-q+1}.
			\end{align*}
			Splitting the last term according to (\ref{jacobi1}) gives
			\begin{align*}
				(p-q)P_{q+1}^{n-2,p-q}(x)+2s\Big(P_{q+1}^{n-2,p-q}(x)\Big)'&=(p-q)P_{q+1}^{n-2,p-q}(x)+\frac{1}{2}(x+1)(n+p)\Big(P_{q+1}^{n-1,p-q}(x)-P_{q+1}^{n-2,p-q+1}(x)\Big)\\
				&=(p-q)P_{q+1}^{n-2,p-q}(x)+\frac{1}{2}(x+1)(n+p)P_{q+1}^{n-1,p-q}(x)\\
				&-\frac{1}{2}(n+p)\Big(2P_{q+1}^{n-2,p-q}(x)-(1-x)P_{q+1}^{n-1,p-q}(x)\Big)\\
				&=(p-q)P_{q+1}^{n-2,p-q}(x)-(n+p)\Big(P_{q+1}^{n-2,p-q}(x)-P_{q+1}^{n-1,p-q}(x)\Big),
			\end{align*}
			where equation (\ref{jacobi2}) was used in the second step. When replacing $(n+p)P_{q+1}^{n-1,p-q}(x)$ with respect to property (\ref{jacobi3}) one gets
			\begin{align*}
				(p-q)P_{q+1}^{n-2,p-q}(x)+2s\Big(P_{q+1}^{n-2,p-q}(x)\Big)'&=-(n+q)P_{q+1}^{n-2,p-q}(x)+(n+p+q+1)P_{q+1}^{n-1,p-q-1}(x)\\
				&-(n+q)P_{q}^{n-1,p-q}(x)\\
				&=-(n+q)\Big(P_{q+1}^{n-2,p-q}(x)+P_{q}^{n-1,p-q}(x)\Big)+(n+p+q+1)P_{q+1}^{n-1,p-q-1}(x).
			\end{align*}
			The result is obtained by applying relation (\ref{jacobi1}), hence
			\begin{align*}
				(p-q)P_{q+1}^{n-2,p-q}(x)+2s\Big(P_{q+1}^{n-2,p-q}(x)\Big)'&=-(n+q)P_{q+1}^{n-1,p-q-1}(x)+(n+p+q+1)P_{q+1}^{n-1,p-q-1}(x)\\
				&=(p+1)P_{q+1}^{n-1,p-q-1}(x).
			\end{align*}
			To show that relation (\ref{jacobi7}) holds we add $(p-q)P_{q+1}^{n-2,p-q}(x)+2s\Big(P_{q+1}^{n-2,p-q}(x)\Big)'-(p+1)P_{q+1}^{n-1,p-q-1}(x)$ which equals
			$0$ due to (\ref{jacobi6}), hence
			\begin{align*}
				(q+1)P_{q+1}^{n-2,p-q}(x)-2s\Big(P_{q+1}^{n-2,p-q}(x)\Big)'&=(q+1)P_{q+1}^{n-2,p-q}(x)-2s\Big(P_{q+1}^{n-2,p-q}(x)\Big)'\\
				&+(p-q)P_{q+1}^{n-2,p-q}(x)+2s\Big(P_{q+1}^{n-2,p-q}(x)\Big)'-(p+1)P_{q+1}^{n-1,p-q-1}(x)\\
				&=(q+1)P_{q+1}^{n-2,p-q}(x)+(p-q)P_{q+1}^{n-2,p-q}(x)-(p+1)P_{q+1}^{n-1,p-q-1}(x)\\
				&=-(p+1)\Big(P_{q+1}^{n-1,p-q-1}(x)-P_{q+1}^{n-2,p-q}(x)\Big)\\
				&=-(p+1)P_{q}^{n-1,p-q}(x),
			\end{align*}
			where we used relation (\ref{jacobi1}) in the last step.\\
			Property (\ref{jacobi8}) follows from (\ref{jacobi4}) by taking into account that $\Gamma(N+1)=N!$ for $N\in\mN$.
		\end{proof}
		\begin{remark}
			\label{pzero}
			Note that the recurrence formulas (\ref{jacobi1}), (\ref{jacobi2}) and (\ref{jacobi3}) of Lemma \ref{JacRec1} as well as (\ref{jacobi6}) and (\ref{jacobi7}) of
			Lemma \ref{JacRec2} also hold formally for $k=-1$ and $q=-1$ respectively by identifying $P_{-1}^{a,b}(x)$ with $0$.
		\end{remark}
		The Gegenbauer polynomials are special cases of the Jacobi polynomials. For a real parameter $\mu>-\frac{1}{2}$ they can be written as
		\begin{align*}
			C_k^\mu(x)=\frac{\big(2\mu\big)_k}{\big(\mu+\frac{1}{2}\big)_k}P_k^{\mu-\frac{1}{2},\mu-\frac{1}{2}}(x).
		\end{align*}
		The Gegenbauer polynomial of degree $k$ and parameter $\mu\in\mR$, $\mu>-\frac{1}{2}$ is given explicitly by
		\begin{align*}
			C_k^\mu(x)=\sum_{j=0}^{\lfloor\frac{k}{2}\rfloor}(-1)^j\frac{\Gamma(k-j+\mu)}{\Gamma(\mu)j!(k-2j)!}(2x)^{k-2j}.
		\end{align*}
		The classic recurrence relations of the Gegenbauer polynomials (cf. \cite{gabor}) are given in the following lemma.
		\begin{lemma}
			For Gegenbauer polynomials of degree $k\in\mN$ and parameter $\mu\in\mR$, $\mu>-\frac{1}{2}$ it holds that
			\begin{align}
				\label{GegenRec}kC_k^\mu(t)&=2(k+\mu-1)tC_{k-1}^\mu(t)-(k+2\mu-2)C_{k-2}^\mu(t)\\
				\label{Gegen1}(k+\mu)C_k^\mu(t)&=\mu\big(C_k^{\mu+1}(t)-C_{k-2}^{\mu+1}(t)\big),\\
				\label{Gegen2}4\mu(\mu+k+1)(1-t^2)C_{k}^{\mu+1}&=(k+2\mu)(k+2\mu+1)C_{k}^{\mu}-(k+1)(k+2)C_{k+2}^{\mu}.
			\end{align}
		\end{lemma}
		As for the Jacobi polynomials we have that the derivative of a Gegenbauer polynomial is again a Gegenbauer polynomial.
		\begin{lemma}
			The derivative of a Gegenbauer polynomial $C_k^\mu(t)$ is given by
			\begin{align}
				\label{Gegen3}\big(C_k^\mu(t)\big)'=2\mu C_{k-1}^{\mu+1}(t).
			\end{align}
		\end{lemma}
		For computations involving Gegenbauer polynomials we will use an additional relation.
		\begin{lemma}
			For a Gegenbauer polynomial $C_k^\mu(t)$ we have that
			\begin{align}
				\label{Gegen4}kC_k^\mu(t)-t\big(C_k^\mu(t)\big)'=-2\mu C_{k-2}^{\mu+1}(t).
				\end{align}
			\end{lemma}
		\begin{proof}
			Using the derivative relation (\ref{Gegen3}) gives
			\begin{align*}
				kC_k^\mu-t\Big(C_k^\mu\Big)'&=kC_k^\mu-2\mu tC_{k-1}^{\mu+1}\\
				&=kC_k^\mu-\mu\Big(\frac{k}{k+\mu}C_k^{\mu+1}+\frac{k+2\mu}{k+\mu}C_{k-2}^{\mu+1}\Big)\\
				&=kC_k^\mu-k\frac{\mu}{k+\mu}\Big(C_k^{\mu+1}-C_{k-2}^{\mu+1}\Big)-2\mu C_{k-2}^{\mu+1},
			\end{align*}
			where also the recurrence relation (\ref{GegenRec}) was used in the second step. When applying equation (\ref{Gegen1}) we get
			\begin{align*}
				kC_k^\mu-t\Big(C_k^\mu\Big)'&=kC_k^\mu-kC_k^\mu-2\mu C_{k-2}^{\mu+1}=-2\mu C_{k-2}^{\mu+1}.
			\end{align*}
		\end{proof}
\section{Spherical monogenics}
	\label{sec:2}
	\subsection{Euclidean Clifford analysis}
		Starting from the real vector space $\mR^m$ with orthonormal basis $\{e_1,\cdots,e_m\}$ the Clifford algebra $\mathcal{C\ell}_m$ can be constructed.
		The (non-commutative) algebraic multiplication follows the rules
		\begin{align*}
			e_je_k+e_ke_j=-2\delta_{jk}\hspace{15pt}j,k=1,\cdots,m.
		\end{align*}
		Elements $X\in\mathcal{C\ell}_m$ can be written in the form
		\begin{align*}
			X=\sum\limits_Ae_AX_A,
		\end{align*}
		with $X_A\in\mC$. Summation runs over all possible ordered index sets $A=\{a_j\}_{j=1}^k$, $1\leq a_1<\cdots<a_k\leq m$, and $e_A$ denotes the 
		$k$-vector $e_A=e_{a_1}\cdots e_{a_k}$.\\
		For $|A|=1$ the Clifford numbers
		\begin{align*}
			\ux=\sum\limits_{j=1}^me_jx_j
		\end{align*}
		correspond to the vectors of $\mathbb{R}^m$. The product of two such Clifford vectors can be written as a sum of a scalar and a bivector, that is
		\begin{align*}
			\ux\uy=\Big(\sum_{j=1}^me_jx_j\Big)\Big(\sum_{k=1}^me_ky_k\Big)=\sum_{j<k}e_je_k(x_jy_k-x_ky_j)-\sum_{j=1}^mx_jy_j=\ux\wedge\uy-\la x,y\ra,
		\end{align*}
		splitting the algebraic product into a wedge and a dot product.	The dual of a Clifford vector takes the form
		\begin{align*}
			\upx=\sum\limits_{j=1}^me_j\partial_{x_j},
		\end{align*}
		which is called the Dirac operator. A null-solution of this operator, i.e. a differentiable function $f:\mathbb{R}^m\rightarrow\mathcal{C\ell}_m$
		satisfying $\upx f(x)=0$, is called (left-)monogenic. Because the Laplace operator $\Delta=-\upx^2$ is factorized
		by this Dirac operator, the theory of monogenic functions, which is called (Euclidean) Clifford analysis, can be considered a refinement of
		harmonic analysis. An introduction to this function theory can be found in \cite{red} and \cite{green}.\\
		\indent Note that $\upx$ can also act from the right on a function,
		in the sense that
		\begin{align*}
			\Big(f(x)\upx\Big)=\sum_{j=1}^m\big(\partial_{x_j}f(x)\big)e_j.
		\end{align*}
		The operators $\ux$ and $\upx$ are invariant under the spin group $Spin(m)$. Satisfying the (anti-commutator) relation
		\begin{align}
			\label{osp12}\big\{\ux,\upx\big\}=\ux\upx+\upx\ux=-2\Big(\mE+\frac{m}{2}\Big),
		\end{align}
		they moreover generate the Lie superalgebra $\mathfrak{osp}(1|2)$ (see \cite{howe}).\\
		\indent By expanding the notation of $k$-homogeneous polynomials $\cP_k$ to $\mathcal{C\ell}_m$-valued polynomials
		\begin{align*}
			P(x)=\sum_Ap_Ax^A, 
		\end{align*}
		with $p_A\in\mathcal{C\ell}_m$ and $|A|=k$, the space of spherical monogenics is defined as follows.
		\begin{definition}
		The space of monogenic $\mathcal{C\ell}_m$-valued polynomials of homogeneity $k$, the so-called spherical monogenics
		of degree $k$, is denoted by $\cM_k=\{M(x)|M\in\cP_k,\upx M(x)=0\}$.
		\end{definition}
		An important property of these spaces is the refinement of the Fischer decomposition (\ref{Fischer}), as any spherical harmonic allows for the decomposition
		\begin{align}
		\label{fischer2}
			\cH_k\otimes\mathcal{C\ell}_m=\cM_k\oplus\ux \cM_{k-1},
		\end{align}
		see e.g. \cite{fischer}. As spherical monogenics are also harmonic they are a special case of spherical harmonics and therefore orthogonal with respect to the 
		$\mathcal{C\ell}_m$-valued inner product on the sphere
		\begin{align*}
		\la P,Q\ra_{\mathbb{S}^{m-1}}=\frac{1}{\omega_{m-1}}\int_{\mathbb{S}^{m-1}}P(x)^\dagger Q(x)d\sigma.
		\end{align*}
		The Clifford conjugation $P(x)^\dagger=\sum\limits_Ae_A^\dagger\overline{p_A}(x)$ satisfies
		\begin{align*}
			e_j^\dagger=-e_j\\
			(XY)^\dagger=Y^\dagger X^\dagger,
		\end{align*}
		for Clifford numbers $X,Y\in\mathcal{C\ell}_m$. One may equally define a $\mathcal{C\ell}_m$-valued	Fischer inner product on $\cP_k$, that is
		\begin{align*}
			\la P,Q\ra_{\partial}=[P(\partial)^\dagger Q(x)]_{x=0},
		\end{align*}
		where $P(\partial)=\sum_Ap_A\partial_{x_1}^{a_1}\cdots\partial_{x_m}^{a_m}$ denotes the operator polynomial that is obtained by replacing $x_j$ with the 
		derivative $\partial_{x_j}$ in $P(x)$. When comparing	these two products one has the same proportionality as before.
		\begin{theorem}
			\label{FischerSphere1}
			For a homogeneous polynomial $P\in\cP_k$ and a spherical harmonic $Q\in\cH_k$ it holds that
			\begin{align*}
				\la P,Q\ra_{\partial}=2^k\Big(\frac{m}{2}\Big)_k\la P,Q\ra_{\mathbb{S}^{m-1}}.
			\end{align*}
			Moreover for two spherical harmonics of arbitrary order $H_k\in\cH_k$ and $Q_l\in\cH_l$, with $k\neq l$, one has that
			\begin{align*}
				2^k\Big(\frac{m}{2}\Big)_k\la H_k,Q_l\ra_{\mS^{m-1}}=\la H_k,Q_l\ra_\partial=0.
			\end{align*}
		\end{theorem}
		\mbox{}\\
		Using the Fischer inner product it is possible to show the duality of the vector $\ux$ and the Dirac operator $\upx$.
		\begin{lemma}
		\label{Duality1}
		On the space of $\mathcal{C\ell}_m$-valued polynomials $\cP$ it holds that
		\begin{align*}
			\la \upx P,Q\ra_{\partial}=-\la P,\ux Q\ra_{\partial}.
		\end{align*}
		\end{lemma}
		\begin{proof}
		Because of linearity it is sufficient to consider a monomial $P(x)=e_Ax_1^{\alpha_1}\cdots x_m^{\alpha_m}\in\cP_{k+1}$ with $\alpha_1+\cdots+\alpha_m=k+1$
		and an arbitrary $k$-homogeneous polynomial $Q(x)\in\cP_k$, hence
		\begin{align*}
		\la P,\ux Q\ra_\partial&=\Big[P(\partial)^\dagger\Big(\ux Q(x)\Big)\Big]_{x=0}\\
			&=\Big[e_A^\dagger\partial_{x_1}^{\alpha_1}\cdots \partial_{x_m}^{\alpha_m}\Big(\sum_{j=1}^me_jx_jQ(x)\Big)\Big]_{x=0}\\
			&=\Big[e_A^\dagger\sum_{j=1}^m\partial_{x_1}^{\alpha_1}\cdots\partial_{x_j}^{\alpha_j-1}\cdots\partial_{x_m}^{\alpha_m}\Big(\alpha_je_jQ(x)\Big)\Big]_{x=0}
			+\Big[e_A^\dagger\ux\big(\partial_{x_1}^{\alpha_1}\cdots\partial_{x_m}^{\alpha_m}Q(x)\big)\Big]_{x=0},
		\end{align*}
		where the last term vanishes and the first term is exactly the Fischer dual of $\upx^\dagger P(x)$, that is
		\begin{align*}
			P(x)^\dagger\upx=e_A^\dagger \sum_{j=1}^m(x_1^{\alpha_1}\cdots x_j^{\alpha_j}\cdots x_m^{\alpha_m})e_j\partial_{x_j}
			=e_A^\dagger \sum_{j=1}^m(x_1^{\alpha_1}\cdots x_j^{\alpha_j-1}\cdots x_m^{\alpha_m})\alpha_je_j,
		\end{align*}
		acting on $Q(x)$, and hence
		\begin{align*}
			\la P,\ux Q\ra_\partial=\la \upx^\dagger P,Q\ra_\partial=-\la\upx P,Q\ra_\partial.
		\end{align*}
		\end{proof}
		\subsection{Reproducing kernel}
		As for the space of spherical harmonics of order $k$ there is also a reproducing kernel for spherical monogenics. This kernel $\widetilde{K}_k^m(x,y)$	can be 
		obtained (up to a constant) by letting two Dirac operators with respect to $\ux$ and $\uy$ act on the reproducing kernel of spherical harmonics $K_{k+1}^m(x,y)$
		from the left and right respectively. We start with the following lemma.
		\begin{lemma}
			\label{diract}
			The action of the Dirac operator $\upx$ on the spherical variable $t=\frac{\la x,y\ra}{|x||y|}$ is
			\begin{align*}
				\upx t=\uy\frac{1}{|x||y|}-\ux\frac{t}{|x|^2}.
			\end{align*}
		\end{lemma}
		The proof of the lemma follows after some computation from the chain rule. Now we can compute the action of the two Dirac operators on the harmonic kernel 
		$K_{k+1}^m$ in two steps, beginning with $\upx$ which results in a vector-valued function.
		\begin{lemma}
			\label{action1}
			Letting $\upx$ act on the harmonic kernel gives
			\begin{align*}
				\upx(K_{k+1}^m)=(m+2k)(\uy|x|^k|y|^kC_k^{\mu+1}(t)-\ux|x|^{k-1}|y|^{k+1}C_{k-1}^{\mu+1}(t)).
			\end{align*}
		\end{lemma}
		\begin{proof}
			When applying the Dirac operator $\upx$ to $K_{k+1}^m(x,y)$ we get
			\begin{align*}
				\upx(K_{k+1}^m(x,y))&=\upx\Big(\frac{\mu+k+1}{\mu}|x|^{k+1}|y|^{k+1}C_{k+1}^{\mu}(t)\Big)\\
				&=\frac{\mu+k+1}{\mu}\Big(\ux|x|^{k-1}|y|^{k+1}(k+1)C_{k+1}^\mu(t)+|x|^{k+1}|y|^{k+1}\Big(\uy\frac{1}{|x||y|}
				-\ux\frac{t}{|x|^2}\Big)\big(C_{k+1}^\mu(t)\big)'	\Big),
			\end{align*}
			where we used Lemma \ref{diract}. Collecting the terms with respect to the vectors $\ux$ and $\uy$ yields
			\begin{align*}
				\upx(K_{k+1}^m(x,y))&=\frac{\mu+k+1}{\mu}\Big(\ux|x|^{k-1}|y|^{k+1}\Big((k+1)C_{k+1}^\mu(t)-t\big(C_{k+1}^\mu(t)\big)'\Big)
				+\uy|x|^k|y|^k\big(C_{k+1}^\mu(t)\big)'\Big),
			\end{align*}
			where we can apply the relations (\ref{Gegen3}) and (\ref{Gegen4})  to get the result.
		\end{proof}
		In the second step we apply $\upy$ from the right to the result of the previous lemma, resulting in a sum of a scalar and a bivector.
		\begin{theorem}
			\label{action2}
			The action of two Dirac operators on the reproducing kernel $K_{k+1}^m(x,y)$ of spherical harmonics yields
			\begin{align*}
				&\upx K_{k+1}^m(x,y)\upy\\
				&=-(m+2k)^2\Big(\frac{2\mu+k}{2\mu}|x|^k|y|^kC_k^\mu(t)+\ux\wedge\uy|x|^{k-1}|y|^{k-1}C_{k-1}^{\mu+1}(t)\Big).
			\end{align*}
		\end{theorem}
		\begin{proof}
			Letting $\upy$ act from the right to the first term of Lemma \ref{action1} gives
			\begin{align*}
				(\uy|x|^k|y|^kC_k^{\mu+1})\upy=-m|x|^k|y|^kC_k^{\mu+1}+\uy\uy|x|^k|y|^{k-2}kC_k^{\mu+1}+\uy|x|^k|y|^k\Big(\ux\frac{1}{|x||y|}
				-\uy\frac{t}{|y|^2}\Big)\big(C_k^{\mu+1}\big)',
			\end{align*}
			where we used Lemma \ref{diract} for the Dirac operator with respect to $\uy$. Collecting the terms with respect to the scalar and bivector part yields
			\begin{align}
				\label{Eterm1}
				(\uy|x|^k|y|^kC_k^{\mu+1})\upy=-(m+k)|x|^k|y|^kC_k^{\mu+1}-2(\mu+1)\ux\wedge\uy|x|^{k-1}|y|^{k-1}C_{k-1}^{\mu+2}.
			\end{align}
			For the action of $\upy$ on the second term of Lemma \ref{action1} we get
			\begin{align*}
				-(\ux|x|^{k-1}|y|^{k+1}C_{k-1}^{\mu+1})\upy=-\ux\uy|x|^{k-1}|y|^{k-1}(k+1)C_{k-1}^{\mu+1}-\ux|x|^{k-1}|y|^{k+1}\Big(\ux\frac{1}{|x||y|}
				-\uy\frac{t}{|y|^2}\Big)\big(C_{k-1}^{\mu+1})'.
			\end{align*}
			Collecting again with respect to the scalar and bivector part results in
			\begin{align}
				-(\ux|x|^{k-1}|y|^{k+1}C_{k-1}^{\mu+1})\upy=&-\ux\wedge\uy|x|^{k-1}|y|^{k-1}\big((k+1)C_{k-1}^{\mu+1}-t\big(C_{k-1}^{\mu+1}\big)'\big)\nonumber\\
				&+|x|^k|y|^k\Big(2(\mu+1)(1-t^2)C_{k-2}^{\mu+2}+t(k+1)C_{k-1}^{\mu+1}\Big)\nonumber\\
				\label{Eterm2}=&-\ux\wedge\uy|x|^{k-1}|y|^{k-1}\big(2C_{k-1}^{\mu+1}-2(\mu+1)C_{k-3}^{\mu+2}\big)\\
				&+|x|^k|y|^k\Big(\frac{(k+2\mu)(k+2\mu+1)}{2(k+\mu)}C_{k-2}^{\mu+1}-\frac{(k-1)k}{2(k+\mu)}C_k^{\mu+1}+\frac{k+1}{2\mu}t\big(C_k^\mu)'\Big),\nonumber
			\end{align}
			where we used relation (\ref{Gegen4}) on the bivector part and relations (\ref{Gegen2}) and (\ref{Gegen3}) on the scalar part.
			By combining equations (\ref{Eterm1}) and (\ref{Eterm2}) we get
			\begin{align*}
				\upx(K_{k+1}^m)\upy&=(m+2k)(\uy|x|^k|y|^kC_k^{\mu+1}-\ux|x|^{k-1}|y|^{k+1}C_{k-1}^{\mu+1})\upy\\
				&=(m+2k)|x|^k|y|^k\bigg(\Big(-\frac{2(k+\mu+1)(k+2\mu)}{2(k+\mu)}-\frac{k(k+1)}{2(k+\mu)}\Big)\Big(C_k^{\mu+1}-C_{k-2}^{\mu+1}\Big)\\
				&-(k+1)C_{k-2}^{\mu+1}+\frac{k+1}{2\mu}t\big(C_k^\mu)'\bigg)\\
				&-(m+2k)\ux\wedge\uy|x|^{k-1}|y|^{k-1}\Big(2C_{k-1}^{\mu+1}+2(\mu+1)\big(C_{k-1}^{\mu+2}-C_{k-3}^{\mu+2}\big)\Big).
			\end{align*}
			Applying relation (\ref{Gegen1}) on both the scalar and bivector part yields
			\begin{align*}
				\upx(K_{k+1}^m)\upy&=(m+2k)|x|^k|y|^k\bigg(-\frac{2(k+\mu+1)(k+2\mu)}{2\mu}C_k^\mu\\
				&-(k+1)C_{k-2}^{\mu+1}-\frac{k+1}{2\mu}\big(kC_k^\mu-t\big(C_k^\mu)'\big)\bigg)\\
				&-(m+2k)\ux\wedge\uy|x|^{k-1}|y|^{k-1}2(\mu+k+1)C_{k-1}^{\mu+1}.
			\end{align*}
			The second and third term of the scalar part cancel each other out due to relation (\ref{Gegen4}). By replacing $2(k+\mu+1)=(m+2k)$ the proof 
			is concluded.
		\end{proof}
		In the following theorem we prove that the result of Theorem \ref{action2} is indeed (up to a constant) the reproducing kernel of spherical monogenics.
		\begin{theorem}
			\label{RepKernel1}
			For any $M_l\in\cM_l$ it holds that
			\begin{align}
			\label{rep1}
			\la\widetilde{K}_k^m(\cdot,y),M_l(\cdot)\ra_{\mS^{m-1}}=\delta_{kl}M_l(y),
			\end{align}
			as well as
			\begin{align}
				\label{rep2}
				\la \widetilde{K}_k^m,\ux M_l\ra_{\mS^{m-1}}=0
			\end{align}
			with the reproducing kernel
			\begin{align*}
				\widetilde{K}_k^m(x,y)&=c_k\upx K_{k+1}^m(x,y)\upy\\
				&=\frac{2\mu+k}{2\mu}|x|^k|y|^kC_k^\mu(t)+(\ux\wedge\uy)|x|^{k-1}|y|^{k-1}C_{k-1}^{\mu+1}(t),
			\end{align*}
			where $c_k=-\frac{1}{(m+2k)^2}$, $\mu=\frac{m}{2}-1$ and $t=\frac{\la x,y\ra}{|x||y|}$.
		\end{theorem}
		\begin{proof}
			For the Fischer inner product of $\widetilde{K}_k^m(x,y)$ and a spherical monogenic $M_l(x)\in\cM_l$ of order $l$ we have
			\begin{align*}
				\la\widetilde{K}_{k}^m,M_l\ra_{\partial}&=\la c_k\upx K_{k+1}^m\upy,M_l\ra_\partial\\
				&=c_k\Big[\Big(\upy^\dagger K_{k+1}^m(x,y)\upx^\dagger\Big)_{x_j\leftrightarrow\partial_{x_j}}M_l(x)\Big]_{x=0}.
			\end{align*}
			Because the inner product is with respect to $x$, the Dirac operator $\upy^\dagger$ can be pulled out. This leads to
			\begin{align*}
				\la\widetilde{K}_{k}^m,M_l\ra_{\partial}&=c_k\upy^\dagger\la(\upx K_{k+1}^m),M_l\ra_\partial\\
				&=-c_k\upy^\dagger\la K_{k+1}^m,\ux M_l\ra_\partial\\
				&=c_k\upy\la K_{k+1}^m,\ux M_l\ra_\partial,
			\end{align*}
			where we made use of the duality of $\upx$ and $\ux$ (cf. Lemma \ref{Duality1}). Using Theorem \ref{FischerSphere1} we can write this 
			in terms of the spherical inner product
			\begin{align*}
				\la\widetilde{K}_{k}^m,M_l\ra_{\partial}&=-c_k2^{k+1}\Big(\frac{m}{2}\Big)_{k+1}\upy\la K_{k+1}^m,\ux M_l\ra_{\mS^{m-1}}.
			\end{align*}
			For a spherical monogenic $M_l(x)$ of order $l$ the function $\ux M_l(x)$ is still harmonic and homogeneous of order $l+1$. It therefore is reproduced
			by the kernel $K_{k+1}^m(x,y)$ for $l=k$, hence
			\begin{align*}
				\la\widetilde{K}_{k}^m,M_l\ra_{\partial}=c_k2^{k+1}\Big(\frac{m}{2}\Big)_{k+1}\upy\Big(\delta_{kl}\uy M_l(y)\Big).
			\end{align*}
			Now letting the Dirac operator $\upy$ act on $\uy M_l(y)$, using (\ref{osp12}), we get
			\begin{align*}
				\la\widetilde{K}_{k}^m,M_l\ra_{\partial}&=c_k2^{k+1}\Big(\frac{m}{2}\Big)_{k+1}\delta_{kl}\big(-mM_l(y)-2\mE M_l(y)-\uy\big(\upy M_l(y)\big)\Big)\\
				&=-c_k2^{k+1}\Big(\frac{m}{2}\Big)_{k+1}(m+2k)\delta_{kl}M_l(y),
			\end{align*}
			where the last equality holds due to the monogeneity and homogeneity of $M_l(x)$.	For the spherical inner product of $\widetilde{K}_k^m(x,y)$
			and $M_l(x)$ we therefore have, using Theorem \ref{FischerSphere1},
			\begin{align*}
				\la \widetilde{K}_k^m,M_l\ra_{\mS^{m-1}}&=-\frac{1}{2^{k}\Big(\frac{m}{2}\Big)_{k}}c_k2^{k+1}\Big(\frac{m}{2}\Big)_{k+1}(m+2k)\delta_{kl}M_l(y)\\
				&=-c_k(m+2k)^2\delta_{kl}M_l(y),
			\end{align*}
			which completes the proof of (\ref{rep1}). For (\ref{rep2}) we proceed as follows.\\
			As before we consider the Fischer inner product of $\widetilde{K}_k^m$ and $\ux M_l$
			\begin{align*}
			\la \widetilde{K}_k^m,\ux M_l\ra_\partial&=c_k\upy^\dagger\la\upx K_{k+1}^m,\ux M_l\ra_\partial\\
				&=-c_k\upy^\dagger\la\upx\upx K_{k+1}^m,M_l\ra_\partial\\
				&=c_k\upy^\dagger\la\Delta K_{k+1}^m,M_l\ra_\partial\\
				&=c_k\upy^\dagger\la 0,M_l\ra_\partial=0,
			\end{align*}
			where we used the duality of $\ux$ and $\upx$, the fact that $-\Delta=\upx^2$ and the harmonicity of $K_{k+1}^m$. The statement then follows by the proportionality of 
			the Fischer and the spherical inner product as stated in Theorem \ref{FischerSphere1}.
		\end{proof}
			
		\begin{remark}
			Note that in \cite{deform} the reproducing kernel on $\cM_k$ is given as
			\begin{align*}
				Z_k(x,y)=\frac{2\mu+k}{2\mu}|x|^k|y|^kC_k^\mu(t)-(\ux\wedge\uy)|x|^{k-1}|y|^{k-1}C_{k-1}^{\mu+1}(t),
			\end{align*}
			with a negative bivector part. The reason for this is the different notation of the reproduction property
			\begin{align*}
				\frac{1}{\omega_{m-1}}\int_{\mS^{m-1}}Z_k(x,y)M_l(x)d\sigma(x)=\delta_{kl}M_l(y),
			\end{align*}
			that differs from ours in the conjugation of the first argument. We have indeed that
			\begin{align*}
				\big(K_k^m(x,y)\big)^\dagger&=\frac{2\mu+k}{2\mu}|x|^k|y|^kC_k^\mu(t)+(\ux\wedge\uy)^\dagger|x|^{k-1}|y|^{k-1}C_{k-1}^{\mu+1}(t)\\
				&=\frac{2\mu+k}{2\mu}|x|^k|y|^kC_k^\mu(t)-(\ux\wedge\uy)|x|^{k-1}|y|^{k-1}C_{k-1}^{\mu+1}(t)=Z_k(x,y).
			\end{align*}
		\end{remark}
\section{Spherical h-monogenics}
\label{sec:3}
	\subsection{Hermitian Clifford analysis}
		Taking the dimension of the underlying vector space to be even, i.e. $m=2n$, we now consider the complex Clifford algebra
		$\mC_{2n}=\mathcal{C\ell}_{2n}\oplus i\mathcal{C\ell}_{2n}$. By applying a complex structure $J\in SO(2n,\mR)$ on the elements of $\mC_{2n}$ one generates the so-called
		Hermitian setting of Clifford analysis, see \cite{fundaments,fundaments2,howe,fischer}. Particularly one chooses $J$ according to its action on the basis vectors 
		$e_1,\cdots,e_{2n}$ as
		\begin{align*}
			J[e_j]=-e_{n+j}\hspace{15pt}J[e_{n+j}]=e_j\hspace{15pt}j=1,\cdots,n.
		\end{align*}
		There are two projection operators $\frac{1}{2}(1\pm iJ)$ related to this complex structure. They act on the basis elements as\\
		\begin{alignat*}{7}
			f_j&=&&\frac{1}{2}(1+iJ)[e_j]&&=&&\frac{1}{2}(e_j-ie_{n+j})&\hspace{15pt}&j=1,\cdots,n\\
			f_j^\dagger&=&-&\frac{1}{2}(1-iJ)[e_{n+j}]&&=&-&\frac{1}{2}(e_j+ie_{n+j})&&j=1,\cdots,n.
		\end{alignat*}
		These new Witt basis elements satisfy the Grassmann identities
		\begin{align*}
			f_jf_k^\dagger+f_k^\dagger f_j=\delta_{jk}\hspace{20pt}j,k=1,\cdots,n
		\end{align*}
		and the duality relations
		\begin{align}
		\label{witt}
			f_jf_k+f_kf_j=f_j^\dagger f_k^\dagger+f_k^\dagger f_j^\dagger=0\hspace{20pt}j,k=1,\cdots,n.
		\end{align}
		When applying these projection operators on the Clifford vectors $\ux\in\mathcal{C\ell}_{2n}$ one obtains the Hermitian Clifford vectors
		\begin{align*}
			\uz=\frac{1}{2}(1+iJ)[\ux]&=\sum\limits_{j=1}^nf_jz_j
		\end{align*}
		as well as their Hermitian conjugated counterparts
		\begin{align*}
			\uz^\dagger=-\frac{1}{2}(1-iJ)[\ux]&=\sum\limits_{j=1}^nf_j^\dagger\zc_j.
		\end{align*}
		Similarly to their action on the Clifford vectors one can apply the two projection operators on the Dirac operator, yielding two new complex conjugated Dirac operators
		\begin{align*}
			\upz&=-\frac{1}{4}(1-iJ)[\upx]=\sum\limits_{j=1}^nf_j^\dagger \partial_{z_j},\\
			\upzd&=\frac{1}{4}(1+iJ)[\upx]=\sum\limits_{j=1}^nf_j\dbar_{z_j}.
		\end{align*}
		Simultaneous null-solutions of these two new operators, i.e. $\mC_{2n}$-valued functions $g(z)$ such that
		\begin{align*}
			\upz g=0=\upzd g,
		\end{align*}
		are now called Hermitian (or h-) monogenic functions.\\
		\begin{remark}
		\label{nil}
			Note that the Hermitian vectors $\uz$, $\uzd$ and Dirac operators $\upz$, $\upzd$ are nilpotent, that is
			\begin{align*}
				\uz\uz=\sum_{j=1}^nf_jz_j\sum_{k=1}^nf_kz_k=\sum_{j=1}^n\sum_{k=1}^n(f_jf_k+f_kf_j)z_jz_k=0,
			\end{align*}
			because of the duality relations (\ref{witt}).
		\end{remark}
		\indent Applying the Hermitian Dirac operator $\upz$ from the left to the vector $\uz$ results in a constant Clifford number called $\beta$.
		\begin{align*}
			\beta&=\upz\uz=\sum_{j=1}^n\sum_{k=1}^nf_j^\dagger\partial_{z_j}f_kz_k=\sum_{j=1}^nf_j^\dagger f_j.
		\end{align*}
		\begin{lemma}
		\label{beta}
		One has the symmetry relations
		\begin{align*}
			\upz\uz&=\Big(\uzd\upzd\Big)=\beta,\\
			\Big(\uz\upz\Big)&=\upzd\uzd=n-\beta,
		\end{align*}
		the commutator relations
		\begin{align}
			\label{comm1}\beta\uz&=\uz(\beta-1),\\
			\nonumber\beta\uzd&=\uzd(\beta+1)
		\end{align}
		as well as the property
		\begin{align}
			\label{betafactor}
			\prod_{j=0}^n(\beta-j)=0.
		\end{align}
		\end{lemma}
		The operators $\uz$, $\uzd$, $\upz$ and $\upzd$ are invariant under the action of the unitary group $U(n)$. They satisfy the (anti-commutator) relations
		\begin{align}
			\label{sl12a}\big\{\uz,\upz\big\}&=\mE_z+\beta,\\
			\label{sl12b}\big\{\uzd,\upzd\big\}&=\mE_{\zc}+n-\beta
		\end{align}
		and generate the Lie superalgebra $\mathfrak{sl}(1|2)$ (see \cite{howe}).\\
		\indent Denoting with $\mS^{(j)}$ the space of $j$-homogenous spinors, we consider the space
		\begin{align*}
			\cP_{p,q}^{(j)}=\cP_{p,q}\otimes\mS^{(j)}
		\end{align*}
		of $\mS^{(j)}$-valued homogeneous polynomials of order $(p,q)$ (see \cite{fundaments2}, \cite{hmonogenics} for a concrete construction).
		\begin{definition}
			The space $\cM_{p,q}^{(j)}=\{M(z)|M(z)\in\cP_{p,q}^{(j)},\upz M(z)=\upzd M(z)=0\}$ of h-monogenic $\mS^{(j)}$-valued polynomials of
			homogeneity $(p,q)$ is called the space of spherical h-monogenics.
		\end{definition}
		Similar to spherical monogenics they allow a further refinement of spherical harmonics by means of a Fischer decomposition (cf. \cite{fischer}),
		hence
		\begin{align*}
			\cH_{p,q}^{(j)}=\cM_{p,q}^{(j)}\oplus\uz\cM_{p-1,q}^{(j-1)}\oplus\uzd\cM_{p,q-1}^{(j+1)}\oplus(c_1\uz\uzd+c_2\uzd\uz)\cM_{p-1,q-1}^{(j)},
		\end{align*}
		where $c_1=\frac{1}{q-1+j}$, $c_2=\frac{1}{p-1+n-j}$ and $\cH_{p,q}^{(j)}=\cH_{p,q}\otimes\mS^{(j)}$ is the space of $\mS^{(j)}$-valued 
		spherical harmonics.\\
		Like in the Euclidean setting we have the $L^2$ inner product
		\begin{align*}
		\la P,Q\ra_{\mathbb{S}^{2n-1}}=\frac{1}{\omega_{2n-1}}\int_{\mathbb{S}^{2n-1}}P(z)^\dagger Q(z)d\sigma(z)
		\end{align*}
		 as well as the Fischer inner product 
		\begin{align*}
		\la P,Q\ra_\partial=\big[P(\partial)^\dagger Q(z)\big]_{z=0}.
		\end{align*}
		As in the case of scalar valued polynomials, $P(\partial)$ is obtained by replacing the complex	variable $z_j=x_j+ix_{n+j}$ with the derivative 
		$2\dbar_{z_j}=\partial_{x_j}+i\partial_{x_{n+j}}$	in $P(z)$. Of special interest is the duality of the vector variables $\uz$ and $\uzd$ with the Dirac 
		operators $\upz$ and $\upzd$ respectively.
		\begin{lemma}
		\label{FischerSphere2}
			If $P(z)$ and $Q(z)$ are homogeneous $\mC_{2n}$-valued polynomials then it holds that
			\begin{align*}
				2\la\upz P,Q\ra_\partial&=\la P,\uz Q\ra_\partial,\\
				2\la\upzd P,Q\ra_\partial&=\la P,\uzd Q\ra_\partial,\\
				\la \uz P,Q\ra_\partial&=2\la P,\upz Q\ra_\partial,\\
				\la \uzd P,Q\ra_\partial&=2\la P,\upzd Q\ra_\partial.
			\end{align*}
		\end{lemma}
	The proof of these dualities follows the same principle as in the Euclidean case in Lemma \ref{Duality1} and is therefore omitted.
	\subsection{Reproducing kernel}
			We apply the two Hermitian Dirac operators with respect to $\uz$ from the left and with respect to $\uu$ from the right
			to the harmonic reproducing kernel, given in Theorem \ref{harmkernel2},
			\begin{align*}
				\label{monokernel}
				K^{n}_{p+1,q+1}(z,u)=c_{p+1,q+1}\la z,u\ra^{p-q}\la z,z\ra^{q+1}\la u,u\ra^{q+1}P_{q+1}^{\nu,p-q}(2s-1)
			\end{align*}
			of homogeneity $(p+1,q+1)$ in $z$ with $p>q\geq 1$, $c_{p+1,q+1}=\frac{\nu+3+p+q}{\nu+1}$ and $\nu=n-2$. The resulting (Clifford-algebra-valued) polynomial
			\begin{align}
			\tilde{K}_{p,q}^n(z,u)=\upzd\upz K_{p+1,q+1}^n(z,u)\upud\upu
			\end{align}
			will be $(p,q)$-homogeneous because the Dirac operators $\upz$ and $\upzd$ reduce the homogeneity of $K_{p+1,q+1}^n$
			with respect to $z$ and $\zc$ respectively by $1$. Moreover we have
			\begin{align*}
				\upz\tilde{K}_{p,q}^n&=\upz\upzd\upz K^{n}_{p+1,q+1}\upud\upu\\
				&=\upz(\upz\upzd+\upzd\upz)K^{n}_{p+1,q+1}\upud\upu\\
				&=\upz(\Delta K^{n}_{p+1,q+1}\upud\upu)=0,
			\end{align*}
			where the second equality holds because of $\upz\upz=0$ (cf. Remark \ref{nil}) and the last one is due to the harmonicity of $K^{n}_{p+1,q+1}$. In a similar
			way we have that $\tilde{K}_{p,q}^n$ is a null-solution of $\upzd$ and hence h-monogenic.	Note also that the order in which the two Dirac operators 
			are applied is not important, as
			\begin{align*}
				&\upzd\upz K^{n}_{p+1,q+1}\upud\upu\\
				&=(\Delta-\upz\upzd)K^{n}_{p+1,q+1}(\Delta-\upu\upud)\\
				&=\upz\upzd K^{n}_{p+1,q+1}\upu\upud.
			\end{align*}
			Our main aim in this section is two-fold: to find an explicit expression for (\ref{monokernel}) in terms of Jacobi polynomials and to show it is the reproducing
			kernel for $\cM_{p,q}^{(j)}$.	For reasons of readability we will use the following notations in the 
			subsequent computations\
			\begin{align*}
				A&=\la z,u\ra=\{\uz,\uud\}\\
				B&=\la u,z\ra=\{\uzd,\uu\}\\
				C&=\la z,z\ra=\{\uz,\uzd\}\\
				D&=\la u,u\ra=\{\uu,\uud\},
			\end{align*}
			with $\{a,b\}=ab+ba$ the anti-commutator. To compute the action of the Dirac operators on the harmonic kernel $K_{p+1,q+1}^n$ it is necessary to know 
			its action on the Jacobi polynomials $P^{a,b}_k(x)=P^{a,b}_k(2s-1)$.
			\begin{lemma}
			\label{DiracP}The four Dirac operators act on the Jacobi polynomials $P^{a,b}_k(x)=P^{a,b}_k(2s-1)$ as
			\begin{align*}
				&\upz P^{a,b}_k(x)=2(\uud\frac{s}{A}-\uzd\frac{s}{C})\Big(P^{a,b}_k(x)\Big)'\\
				&\upzd P^{a,b}_k(x)=2(\uu\frac{s}{B}-\uz\frac{s}{C})\Big(P^{a,b}_k(x)\Big)'\\
				&\upu P^{a,b}_k(x)=2(\uzd\frac{s}{B}-\uud\frac{s}{D})\Big(P^{a,b}_k(x)\Big)'\\
				&\upud P^{a,b}_k(x)=2(\uz\frac{s}{A}-\uud\frac{s}{D})\Big(P^{a,b}_k(x)\Big)',
			\end{align*}
			where $x=2s-1$ with $s=\frac{AB}{CD}$.
			\end{lemma}
			\begin{proof}
			Letting $\upz$ act on the Hermitian inner product $A=\la z,u\ra$ one gets
			\begin{align*}
				\upz A=\upz\la z,u\ra=\sum_{j=1}^nf_j^\dagger\partial_{z_j}\sum_{l=1}^nz_l\uc_l
				=\sum_{j=1}^nf_j^\dagger\partial_{z_j}z_j\uc_j=\sum_{j=1}^nf_j^\dagger\uc_j=\uud.
			\end{align*}
			Equivalently one has
			\begin{align*}
				\upud A&=\uz&&\\
				\upzd B&=\uu&&\upu B=\uzd\\
				\upz C&=\uzd&&\upzd C=\uz\\
				\upu D&=\uud&&\upud D=\uu.
			\end{align*}
			Applying the quotient rule to the angular variable $s=\frac{\la z,u\ra\la u,z\ra}{\la z,z\ra\la u,u\ra}=\frac{AB}{CD}$ one gets
			\begin{align*}
				\upz s=\upz\frac{AB}{CD}=\frac{\uud BCD-\uzd ABD}{C^2D^2}=\uud\frac{s}{A}-\uzd\frac{s}{C}
			\end{align*}
			and furthermore
			\begin{align*}
				\upzd s&=\uu\frac{s}{B}-\uz\frac{s}{C}\\
				\upu s&=\uzd\frac{s}{B}-\uud\frac{s}{D}\\
				\upud s&=\uz\frac{s}{A}-\uu\frac{s}{D}.
			\end{align*}
			The claim follows by the chain rule.
			\end{proof}
			\mbox{}\newline
			To compute $\tilde{K}_{p,q}^n$ we will now apply the four necessary Dirac operators consecutively to the harmonic kernel $K_{p+1,q+1}^n$. For the sake of
			readability these results are collected in the following three lemmas, after which the final computation is obtained in Theorem \ref{Dirac4}. Note that we omit the
			arguments of the Jacobi polynomials.
			\begin{lemma}
			\label{Dirac1}
				The Dirac operator $\upz$ acts on the harmonic kernel as
				\begin{align*}
				&\upz K_{p+1,q+1}^n=c_{p+1,q+1}(p+1)\bigg(\uud A^{p-q-1}C^{q+1}D^{q+1}P_{q+1}^{n-1,p-q-1}-\uzd A^{p-q}C^{q}D^{q+1}P_{q}^{n-1,p-q}\bigg).\\
			\end{align*}
			\end{lemma}
			\begin{proof}
			Using the product rule and the results of Lemma \ref{DiracP} we get
			\begin{align*}
				\upz K_{p+1,q+1}^n&=c_{p+1,q+1}\upz\Big(A^{p-q}C^{q+1}D^{q+1}P_{q+1}^{n-2,p-q}\Big)\\
				&=c_{p+1,q+1}\Big(\uud A^{p-q-1}C^{q+1}D^{q+1}(p-q)P_{q+1}^{n-2,p-q}\\
				&+\uzd A^{p-q}C^{q}D^{q+1}(q+1)P_{q+1}^{n-2,p-q}\\
				&+(\uud\frac{s}{A}-\uzd\frac{s}{C})A^{p-q}C^{q+1}D^{q+1}2\Big(P_{q+1}^{n-2,p-q}\Big)'\Big).
			\end{align*}
			The resulting terms can be collected with regard to the vectors $\uud$ and $\uzd$, yielding
			\begin{align*}
				\upz K_{p+1,q+1}^n&=c_{p+1,q+1}\bigg(\uud A^{p-q-1}C^{q+1}D^{q+1}\bigg((p-q)P_{q+1}^{n-2,p-q}+2s\Big(P_{q+1}^{n-2,p-q}\Big)'\bigg)\\
				&+\uzd A^{p-q}C^{q}D^{q+1}\bigg((q+1)P_{q+1}^{n-2,p-q}-2s\Big(P_{q+1}^{n-2,p-q}\Big)'\bigg)\bigg)
			\end{align*}
			and the statement then follows by applying the recurrence formulas (\ref{jacobi6}) and (\ref{jacobi7}) of Lemma \ref{JacRec2} respectively.
			\end{proof}
			\indent Applying the vector-valued Dirac operator $\upz$ to the scalar polynomial $K_{p+1,q+1}^n$ results naturally in a vector-valued polynomial. 
			By applying a second Dirac opartor $\upzd$ in the next step we expect the result to be a sum of scalars and bivectors, as confirmed in the following lemma.
			\begin{lemma}
			\label{Dirac2}
				The two Dirac operators $\upzd\upz$ act on the harmonic kernel as
				\begin{align*}
					\upzd\upz K_{p+1,q+1}^n=c_{p+1,q+1}(p+1)&\bigg(\uz\uzd A^{p-q}C^{q-1}D^{q+1}pP_{q-1}^{n,p-q}+\uu\uud A^{p-q}C^{q}D^{q}(n+p)P_{q}^{n,p-q}\\
					&-\uz\uud A^{p-q-1}C^{q}D^{q+1}pP_{q}^{n,p-q-1}-\uu\uzd A^{p-q+1}C^{q-1}D^{q}(n+p)P_{q-1}^{n,p-q+1}\\
					&-(n-\beta)A^{p-q}C^{q}D^{q+1}P_{q}^{n-1,p-q}\bigg).
				\end{align*}
			\end{lemma}
			\begin{proof}
				We act with the conjugated Dirac operator $\upzd$ from the left on the result of Lemma \ref{Dirac1}, which leads to
				\begin{align*}
					\upzd\bigg(\upz K_{p+1,q+1}^n\bigg)&=c_{p+1,q+1}(p+1)\upzd\bigg(\uud A^{p-q-1}C^{q+1}D^{q+1}P_{q+1}^{n-1,p-q-1}-\uzd A^{p-q}C^{q}D^{q+1}P_{q}^{n-1,p-q}\bigg)\\
					&=c_{p+1,q+1}(p+1)\bigg(\uz\uud A^{p-q-1}C^{q}D^{q+1}(q+1)P_{q+1}^{n-1,p-q-1}\\
					&+(\uu\frac{s}{B}-\uz\frac{s}{C})\uud A^{p-q-1}C^{q+1}D^{q+1}2\Big(P_{q+1}^{n-1,p-q-1}\Big)'-(n-\beta)A^{p-q}C^{q}D^{q+1}P_{q}^{n-1,p-q}\\
					&-\uz\uzd A^{p-q}C^{q-1}D^{q+1}qP_{q}^{n-1,p-q}-(\uu\frac{s}{B}-\uz\frac{s}{C})\uzd A^{p-q}C^{q}D^{q+1}2\Big(P_{q}^{n-1,p-q}\Big)'\bigg),
				\end{align*}
				where again Lemma \ref{JacRec2} was used. The resulting terms are then collected with respect to $\uz\uzd$, $\uu\uud$, $\uz\uud$, $\uu\uzd$ and $(n-\beta)$
				\begin{align*}
					\upzd\upz K_{p+1,q+1}^n&=c_{p+1,q+1}(p+1)\bigg(\uz\uud A^{p-q-1}C^{q}D^{q+1}\Big((q+1)P_{q+1}^{n-1,p-q-1}-2s\big(P_{q+1}^{n-1,p-q-1}\big)'\Big)\\
					&-\uz\uzd A^{p-q}C^{q-1}D^{q+1}\Big(qP_{q}^{n-1,p-q}-2s\big(P_{q}^{n-1,p-q}\big)'\Big)+\uu\uud A^{p-q}C^{q}D^{q}2\big(P_{q+1}^{n-1,p-q-1}\big)'\\
					&-\uu\uzd A^{p-q+1}C^{q-1}D^{q}2\big(P_{q}^{n-1,p-q}\big)'-(n-\beta)A^{p-q}C^{q}D^{q+1}P_{q}^{n-1,p-q}\bigg).
				\end{align*}
				The statement follows by applying formula (\ref{jacobi7}) of Lemma \ref{JacRec2} to the first two terms and the derivative relation (\ref{jacobi5}) of 
				Lemma \ref{JacRec1} to the third and fourth term.
				\end{proof}
			\mbox{}\\
			In the third step the conjugated Dirac operator $\upud$ acts from the right on the result of the previous lemma.
			\begin{lemma}
			\label{Dirac3}
				Letting the operator $\upzd\upz$ act from the left and the Dirac operator $\upud$ from the right on the harmonic kernel one gets
				\begin{align*}
					\bigg(\upzd\upz K_{p+1,q+1}^n\bigg)\upud&=c_{p+1,q+1}(p+1)(n+p+q+1)\bigg(A^{p-q}C^{q}D^{q}(\beta+p)\uu P_{q}^{n-1,p-q}\\
					&-A^{p-q-1}C^{q}D^{q}\uz\uud\uu pP_{q}^{n,p-q-1}+A^{p-q}C^{q-1}D^{q}\uz\uzd\uu pP_{q-1}^{n,p-q}\bigg).
				\end{align*}
			\end{lemma}
			\begin{proof}
			As in the previous lemmata we compute the action of the Dirac operator on the terms of the previous lemma's result, collect the outcome in terms
			of scalars, vectors, bivectors and 3-vectors and apply the recurrence formulas for the Jacobi polynomials to get the final result.\\
			\indent We apply the Dirac operator $\upud$ from the right on the first term of Lemma \ref{Dirac2}, leading to
			\begin{align*}
				\big(\uz\uzd A^{p-q}C^{q-1}D^{q+1}pP_{q-1}^{n,p-q}\big)\upud&=\uz\uzd\uz A^{p-q-1}C^{q-1}D^{q+1}p(p-q)P_{q-1}^{n,p-q}\\
				&+\uz\uzd\uu A^{p-q}C^{q-1}D^{q}p(q+1)P_{q-1}^{n,p-q}\\
				&+\uz\uzd A^{p-q}C^{q-1}D^{q+1}2p(\uz\frac{s}{A}-\uu\frac{s}{D})\big(P_{q-1}^{n,p-q}\big)'.
			\end{align*}
			Because of the anti-commutator relation $\uz\uzd\uz=\uz(C-\uz\uzd)=\uz C$ this yields
			\begin{align*}
				\big(\uz\uzd A^{p-q}C^{q-1}D^{q+1}pP_{q-1}^{n,p-q}\big)\upud&=\uz A^{p-q-1}C^{q}D^{q+1}p\Big((p-q)P_{q-1}^{n,p-q}+2s\big(P_{q-1}^{n,p-q}\big)'\Big)\\
				&+\uz\uzd\uu A^{p-q}C^{q-1}D^{q}p\Big((q+1)P_{q-1}^{n,p-q}-2s\big(P_{q-1}^{n,p-q}\big)'\Big).
			\end{align*}
			Applying the recurrence formula (\ref{jacobi5}) of Lemma \ref{JacRec2} to the first term and formula (\ref{jacobi6}) to the second results in
			\begin{align}
				\label{term1}\big(\uz\uzd A^{p-q}C^{q}D^{q+1}pP_{q-1}^{n,p-q}\big)\upud&=\uz A^{p-q-1}C^{q-1}D^{q+1}p(p-1)P_{q-1}^{n+1,p-q-1}\\
				&+\uz\uzd\uu A^{p-q}C^{q-1}D^{q}p\big(2P_{q-1}^{n,p-q}-(p-1)P_{q-2}^{n+1,p-q}\big).\nonumber
			\end{align}
			Acting with $\upud$ on the second term of Lemma \ref{Dirac2} gives
			\begin{align*}
				\big(\uu\uud A^{p-q}C^qD^q(n+p)P_q^{n,p-q}\big)\upud&=\uu\beta A^{p-q}C^qD^q(n+p)P_q^{n,p-q}\\
				&+\uu\uud\uz A^{p-q-1}C^qD^q(p-q)(n+p)P_q^{n,p-q}\\
				&+\uu\uud\uu A^{p-q}C^qD^{q-1}q(n+p)P_q^{n,p-q}\\
				&+\uu\uud A^{p-q}C^qD^q2(n+p)(\uz\frac{s}{A}-\uu\frac{s}{D})\big(P_q^{n,p-q}\big)'.
			\end{align*}
			Using once more $\uu\uud\uu=\uu(D-\uu\uud)=\uu D$ we have
			\begin{align*}
				\big(\uu\uud A^{p-q}C^qD^q(n+p)P_q^{n,p-q}\big)\upud&=\uu A^{p-q}C^qD^q(n+p)\Big((q+\beta)P_q^{n,p-q}-2s\big(P_q^{n,p-q}\big)'\\
				&+\uu\uud\uz A^{p-q-1}C^qD^q(n+p)\Big((p-q)P_q^{n,p-q}+2s\big(P_q^{n,p-q}\big)'\Big).
			\end{align*}
			As before we apply formula (\ref{jacobi5}) of Lemma \ref{JacRec2} to the second term and formula (\ref{jacobi6}) to the first, hence
			\begin{align*}
				\big(\uu\uud A^{p-q}C^qD^q(n+p)P_q^{n,p-q}\big)\upud&=\uu A^{p-q}C^qD^q(n+p)\big(\beta P_q^{n,p-q}-pP_{q-1}^{n+1,p-q}\big)\\
				&+(\uu A+\uz D-\uz\uud\uu)A^{p-q-1}C^qD^q(n+p)pP_q^{n+1,p-q-1}\\
				&=\uu A^{p-q}C^qD^q(n+p)\big(\beta P_q^{n,p-q}-p(P_{q-1}^{n+1,p-q}-P_q^{n+1,p-q-1})\big)\\
				&+\uz A^{p-q-1}C^qD^{q+1}(n+p)pP_q^{n+1,p-q-1}\\
				&-\uz\uud\uu A^{p-q-1}C^qD^q(n+p)pP_q^{n+1,p-q-1},
			\end{align*}
			where we also used $\uu\uud\uz=\uu(A-\uz\uud)=\uu A+\uz\uu\uud=\uu A+\uz(D-\uud\uu)$. Applying formula (\ref{jacobi1}) of Lemma \ref{JacRec1} to the first term results in
			\begin{align}
				\label{term2}\big(\uu\uud A^{p-q}C^qD^q(n+p)P_q^{n,p-q}\big)\upud&=\uu A^{p-q}C^qD^q(n+p)(\beta+p)P_q^{n,p-q}\\
				&+\uz A^{p-q-1}C^qD^{q+1}(n+p)pP_q^{n+1,p-q-1}\nonumber\\
				&-\uz\uud\uu A^{p-q-1}C^qD^q(n+p)pP_q^{n+1,p-q-1}.\nonumber
			\end{align}
			Letting $\upud$ act on the third term of Lemma \ref{Dirac2} gives
			\begin{align*}
				\big(-\uz\uud A^{p-q-1}C^qD^{q+1}pP_q^{n,p-q-1}\big)\upud&=-\uz\beta A^{p-q-1}C^qD^{q+1}pP_q^{n,p-q-1}\\
				&-\uz\uud\uz A^{p-q-2}C^qD^{q+1}p(p-q-1)P_q^{n,p-q-1}\\
				&-\uz\uud\uu A^{p-q-1}C^qD^{q}p(q+1)P_q^{n,p-q-1}\\
				&-\uz\uud A^{p-q-1}C^qD^{q+1}2p(\uz\frac{s}{A}-\uu\frac{s}{D})\big(P_q^{n,p-q-1}\big)'.
			\end{align*}
			Using $\uz\uud\uz=\uz(A-\uz\uud)=\uz A$ and collecting in terms of $\uz$ and $\uz\uud\uu$ we have
			\begin{align*}
				\big(-\uz\uud A^{p-q-1}C^qD^{q+1}pP_q^{n,p-q-1}\big)\upud&=-\uz A^{p-q-1}C^qD^{q+1}p\Big((\beta+p-q-1)P_q^{n,p-q-1}+2s\big(P_q^{n,p-q-1})'\Big)\\
				&-\uz\uud\uu A^{p-q-1}C^qD^{q}p\Big((q+1)P_q^{n,p-q-1}-2s\big(P_q^{n,p-q-1}\big)'\Big),
			\end{align*}
			where we again can apply formulas (\ref{jacobi5}) and (\ref{jacobi6}) of Lemma \ref{JacRec2}, yielding
			\begin{align}
				\label{term3}\big(-\uz\uud A^{p-q-1}C^qD^{q+1}pP_q^{n,p-q-1}\big)\upud&=-\uz A^{p-q-1}C^qD^{q+1}p\big(\beta P_q^{n,p-q-1}+(p-1)P_q^{n+1,p-q-2}\big)\\
				&-\uz\uud\uu A^{p-q-1}C^qD^{q}p\big(P_q^{n,p-q-1}-(p-1)P_{q-1}^{n+1,p-q-1}\big).\nonumber
				\end{align}
			For the fourth term of Lemma \ref{Dirac2} we get
			\begin{align*}
				\big(-\uu\uzd A^{p-q+1}C^{q-1}D^q(n+p)P_{q-1}^{n,p-q+1}\big)\upud&=-\uu\uzd\uz A^{p-q}C^{q-1}D^q(n+p)(p-q+1)P_{q-1}^{n,p-q+1}\\
				&-\uu\uzd\uu A^{p-q+1}C^{q-1}D^{q-1}(n+p)qP_{q-1}^{n,p-q+1}\\
				&-\uu\uzd A^{p-q+1}C^{q-1}D^q2(n+p)(\uz\frac{s}{A}-\uu\frac{s}{D})\big(P_{q-1}^{n,p-q+1}\big)'.
			\end{align*}
			With $\uu\uzd\uu=\uu B$ we have
			\begin{align*}
				\big(-\uu\uzd A^{p-q+1}C^{q-1}D^q(n+p)P_{q-1}^{n,p-q+1}\big)\upud&=-\uu\uzd\uz A^{p-q}C^{q-1}D^q(n+p)\Big((p-q+1)P_{q-1}^{n,p-q+1}
				+2s\big(P_{q-1}^{n,p-q+1}\big)'\Big)\\
				&-\uu A^{p-q}C^qD^q(n+p)s\Big(qP_{q-1}^{n,p-q+1}-2s\big(P_{q-1}^{n,p-q+1}\big)'\Big),
			\end{align*}
			where we apply formulas (\ref{jacobi5}) and (\ref{jacobi6}) of Lemma \ref{JacRec2} and $\uu\uzd\uz=\uu C+\uz B-\uz\uzd\uu$ to get
			\begin{align}
				\label{term4}\big(-\uu\uzd A^{p-q+1}C^{q-1}D^q(n+p)P_{q-1}^{n,p-q+1}\big)\upud&=\uz\uzd\uu A^{p-q}C^{q-1}D^q(n+p)pP_{q-1}^{n+1,p-q}\\
				&-\uu A^{p-q}C^qD^q(n+p)\big(sP_{q-1}^{n,p-q+1}+p(P_{q-1}^{n+1,p-q}-sP_{q-2}^{n+1,p-q+1})\big)\nonumber\\
				&-\uz A^{p-q-1}C^qD^{q+1}(n+p)psP_{q-1}^{n+1,p-q}.\nonumber
			\end{align}
			If we let $\upud$ act on the last term of Lemma \ref{Dirac2} we get
			\begin{align*}
				\big((n-\beta)A^{p-q}C^qD^{q+1}P_q^{n-1,p-q}\big)\upud&=-(n-\beta)\uz A^{p-q-1}C^qD^{q+1}(p-q)P_q^{n-1,p-q}\\
				&-(n-\beta)\uu A^{p-q}C^qD^{q}(q+1)P_q^{n-1,p-q}\\
				&-(n-\beta)A^{p-q}C^qD^{q+1}2(\uz\frac{s}{A}-\uu\frac{s}{D})\big(P_q^{n-1,p-q}\big)'.
			\end{align*}
			Collecting the result in terms of $\uu$ and $\uz$ gives
			\begin{align*}
				\big((n-\beta)A^{p-q}C^qD^{q+1}P_q^{n-1,p-q}\big)\upud&=-(n-\beta)\uz A^{p-q-1}C^qD^{q+1}\Big((p-q)P_q^{n-1,p-q}+2s\big(P_q^{n-1,p-q}\big)'\Big)\\
				&-(n-\beta)\uu A^{p-q}C^qD^{q}\Big((q+1)P_q^{n-1,p-q}-2s\big(P_q^{n-1,p-q}\big)'\Big),
			\end{align*}
			where we can apply the recurrence formulas (\ref{jacobi5}) and (\ref{jacobi6}) of Lemma \ref{JacRec2} once more to get
			\begin{align}
				\label{term5}\big((n-\beta)A^{p-q}C^qD^{q+1}P_q^{n-1,p-q}\big)\upud&=-(n-\beta)\uz A^{p-q-1}C^qD^{q+1}pP_q^{n,p-q-1}\\
				&-(n-\beta)\uu A^{p-q}C^qD^{q}\big(P_q^{n-1,p-q}-pP_{q-1}^{n,p-q}\big).\nonumber
			\end{align}
			\indent Collecting the $\uz$-parts from all five terms (\ref{term1}) - (\ref{term5}) results in
			\begin{align*}
				&\uz A^{p-q-1}C^qD^{q+1}p\Big((p-1)P_{q-1}^{n+1,p-q-1}+(n+p)P_q^{n+1,p-q-1}-\beta P_q^{n,p-q-1}\\
				&-(p-1)P_q^{n+1,p-q-2}-(n+p)sP_{q-1}^{p-q}-(n-\beta+1)P_q^{n,p-q-1}\Big).
			\end{align*}
			We will now show that this term vanishes by denoting
			\begin{align*}
				G(s)&=(p-1)P_{q-1}^{n+1,p-q-1}+(n+p)P_q^{n+1,p-q-1}-\beta P_q^{n,p-q-1}-(p-1)P_q^{n+1,p-q-2}\\
				&-(n+p)sP_{q-1}^{p-q}-(n-\beta+1)P_q^{n,p-q-1}\\
				&=(p-1)(P_{q-1}^{n+1,p-q-1}-P_q^{n+1,p-q-2})+(n+p)P_q^{n+1,p-q-1}-(n+1)P_q^{n,p-q-1}-2s\big(P_q^{n,p-q-1}\big)',
			\end{align*}
			where we substituted $(n+p)sP_{q-1}^{n+1,p-q}=2s\big(P_q^{n,p-q-1}\big)'$ according to the derivative relation 
			(\ref{jacobi5}) of Lemma \ref{JacRec1} and used	the commutator relations of Lemma \ref{beta}.	For the first term we now 
			can apply equation (\ref{jacobi1}) of Lemma \ref{JacRec1} and by subtracting and adding $qP_q^{n,p-q-1}$ we get
			\begin{align*}
				G(s)=-(n+p+q)P_q^{n,p-q-1}+qP_q^{n,p-q-1}-2s\big(P_q^{n,p-q-1}\big)'+(n+p)P_q^{n+1,p-q-1}.
			\end{align*}
			Once again we apply the recurrence relation (\ref{jacobi6}) on the second and third term, yielding
			\begin{align*}
			G(s)=-(n+p+q)P_q^{n,p-q-1}-(p-1)P_{q-1}^{n+1,p-q-1}+(n+p)P_q^{n+1,p-q-1}.
			\end{align*}
			Splitting the new term as $(p-1)P_{q-1}^{n+1,p-q-1}=\big((n+p+q)-(n+q+1)\big)P_{q-1}^{n+1,p-q-1}$ results in
			\begin{align*}
				G(s)=-(n+p+q)P_q^{n,p-q-1}-(n+p+q)P_{q-1}^{n+1,p-q-1}+(n+p)P_q^{n+1,p-q-1}+(n+q+1)P_{q-1}^{n+1,p-q-1}\Big),
			\end{align*}
			which allows us to use equation (\ref{jacobi3}) of Lemma \ref{JacRec1} on the last two terms to get
			\begin{align*}
				G(s)=P_q^{n+1,p-q-2}-P_q^{n,p-q-1}-P_{q-1}^{n+1,p-q-1}.
			\end{align*}
			Using equation (\ref{jacobi1}) of the same lemma results in
			\begin{align}
				\label{zterms}G(s)=P_{q-1}^{n+1,p-q-1}-P_{q-1}^{n+1,p-q-1},
			\end{align}
			which equals 0.\\
			\indent Collecting all $\uu$-terms from equations (\ref{term2}), (\ref{term4}) and (\ref{term5}) gives
			\begin{align*}
				\uu A^{p-q}C^qD^q&\Big((n+p)((\beta+p)P_q^{n,p-q}-pP_{q-1}^{n+1,p-q})\\
				&-(n+p)s\big(P_{q-1}^{n,p-q+1}-pP_{q-2}^{n+1,p-q+1}\big)\\
				&-(n-\beta+1)\big(P_q^{n-1,p-q}-pP_{q-1}^{n,p-q}\big)\Big)\\
				=\uu A^{p-q}C^qD^q&\bigg((n+p)\Big(p(P_q^{n,p-q}-P_{q-1}^{n+1,p-q})-s(P_{q-1}^{n,p-q+1}-pP_{q-2}^{n+1,p-q+1})\Big)
				-(n+1)\big(P_q^{n-1,p-q}-pP_{q-1}^{n,p-q}\big)\\
				&+\beta\Big((n+p)P_q^{n,p-q}+P_q^{n-1,p-q}-pP_{q-1}^{n,p-q}\Big)\bigg)\\
				=\uu A^{p-q}C^qD^q&\Big(G_1(s)+\beta G_2(s)\Big),
			\end{align*}
			where we collected with regard to $\beta$ and denoted
			\begin{align*}
				G_1(s)&=(n+p)\Big(p(P_q^{n,p-q}-P_{q-1}^{n+1,p-q})-s(P_{q-1}^{n,p-q+1}-pP_{q-2}^{n+1,p-q+1})\Big)
				-(n+1)\big(P_q^{n-1,p-q}-pP_{q-1}^{n,p-q}\big),\\
				G_2(s)&=(n+p)P_q^{n,p-q}+P_q^{n-1,p-q}-pP_{q-1}^{n,p-q}.
			\end{align*}
			We show subsequently that
			\begin{align*}
				G_1(s)&=(p-1)(n+p+q+1)P_q^{n-1,p-q},\\
				G_2(s)&=(n+p+q+1)P_q^{n-1,p-q},
			\end{align*}
			beginning with
			\begin{align*}
				G_1(s)&=(n+p)\Big(p(P_q^{n,p-q}-P_{q-1}^{n+1,p-q})\Big)-2s\Big(\big(P_q^{n-1,p-q}\big)'-p\big(P_{q-1}^{n,p-q}\big)'\Big)
				-(n+1)\big(P_q^{n-1,p-q}-pP_{q-1}^{n,p-q}\big),
			\end{align*}
			where we used the derivative relation (\ref{jacobi5}) on the term $(n+p)s(P_{q-1}^{n,p-q+1}-pP_{q-2}^{n+1,p-q+1})$. By
			adding and subtracting the correspondent terms $qP_q^{n-1,p-q}-p(q-1)P_{q-1}^{n,p-q}$ in order to use relation 
			(\ref{jacobi7}) we get
			\begin{align*}
				G_1(s)&=(n+p)\Big(p(P_q^{n,p-q}-P_{q-1}^{n+1,p-q})\Big)+(qP_q^{n-1,p-q}-p(q-1)P_{q-1}^{n,p-q})
				-2s\Big(\big(P_q^{n-1,p-q}\big)'-p\big(P_{q-1}^{n,p-q}\big)'\Big)\\
				&-(n+1)\big(P_q^{n-1,p-q}-pP_{q-1}^{n,p-q}\big)-(qP_q^{n-1,p-q}-p(q-1)P_{q-1}^{n,p-q})\\
				&=(n+p)\Big(p(P_q^{n,p-q}-P_{q-1}^{n+1,p-q})\Big)-(pP_{q-1}^{n,p-q}-p(p-1)P_{q-2}^{n+1,p-q})\\
				&-(n+1)\big(P_q^{n-1,p-q}-pP_{q-1}^{n,p-q}\big)-(qP_q^{n-1,p-q}-p(q-1)P_{q-1}^{n,p-q}).
			\end{align*}
			Splitting $P_q^{n-1,p-q+1}$ and $P_{q-1}^{n,p-q}$ according to (\ref{jacobi1}) gives
			\begin{align*}
				G_1(s)&=p(n+p)\Big(P_q^{n-1,p-q+1}+P_{q-1}^{n,p-q+1}\Big)-p(n+p)P_{q-1}^{n+1,p-q}+p(p-1)P_{q-2}^{n+1,p-q}
				-(n+q+1)P_{q}^{n-1,p-q}\\
				&+p(n+q-1)\Big(P_{q-1}^{n-1,p-q+1}+P_{q-2}^{n,p-q+1}\Big)\\
				&=p\Big((n+p)P_q^{n-1,p-q+1}+(n+q-1)P_{q-1}^{n-1,p-q+1}\Big)+p(n+p)\Big(P_{q-1}^{n,p-q+1}-P_{q-1}^{n+1,p-q}\Big)\\
				&+p(p-1)P_{q-2}^{n+1,p-q}-(n+q+1)P_{q}^{n-1,p-q}+p(n+q-1)P_{q-2}^{n,p-q+1}.
			\end{align*}
			We apply (\ref{jacobi3}) to the first term and add and subtract $pP_q^{n-1,p-q}$ to get
			\begin{align*}
				G_1(s)&=p(n+p+q+1)P_q^{n-1,p-q}+p\Big((n+p)P_{q-1}^{n,p-q+1}+(n+q-1)P_{q-2}^{n,p-q+1}\Big)\\
				&-p(n+p)P_{q-1}^{n+1,p-q}+p(p-1)P_{q-2}^{n+1,p-q}-(n+p+q+1)P_{q}^{n-1,p-q}.
			\end{align*}
			This allows us to use relation (\ref{jacobi3}) on the second term, hence
			\begin{align*}
				G_1(s)&=(p-1)(n+p+q+1)P_q^{n-1,p-q}+p\Big((n+p+q-1)P_{q-1}^{n,p-q}\\
				&-(n+p)P_{q-1}^{n+1,p-q}+p(p-1)P_{q-2}^{n+1,p-q}\Big).
			\end{align*}
			Applying relation (\ref{jacobi3}) to the third term $(n+p)P_{q-1}^{n+1,p-q}$ yields
			\begin{align*}
				G_1(s)&=(p-1)(n+p+q+1)P_q^{n-1,p-q}+p(n+p+q-1)\Big(P_{q-1}^{n,p-q}-P_{q-1}^{n+1,p-q+1}+P_{q-2}^{n+1,p-q}\Big)\\
				&=(p-1)(n+p+q+1)P_q^{n-1,p-q},
			\end{align*}
			where the last term vanished due to (\ref{jacobi1}). For the second part $G_2(s)$ we get
			\begin{align*}
				G_2(s)&=(n+p)P_q^{n,p-q}+P_q^{n-1,p-q}-pP_{q-1}^{n,p-q}\\
				&=(n+p+q+1)P_q^{n-1,p-q}+(n+p)P_q^{n,p-q}-(n+p+q)P_q^{n-1,p-q}-pP_{q-1}^{n,p-q},
			\end{align*}
			where we added and subtracted $(n+p+q)P_q^{n-1,p-q}$. Splitting $(n+p+q)P_q^{n-1,p-q}$ according to relation (\ref{jacobi3}) yields
			\begin{align*}
				G_2(s)&=(n+p+q+1)P_q^{n-1,p-q}+(n+p)\Big(P_q^{n,p-q}-P_q^{n-1,p-q+1}\Big)+(n+q+1)P_{q-1}^{n-1,p-q+1}-pP_{q-1}^{n,p-q}\\
				&=(n+p+q+1)P_q^{n-1,p-q}+(n+p)P_{q-1}^{n,p-q+1}+(n+q+1)\Big(P_{q-2}^{n,p-q+1}-P_{q-1}^{n,p-q}\Big)-pP_{q-1}^{n,p-q}\\
				&=(n+p+q+1)P_q^{n-1,p-q}+(n+p)P_{q-1}^{n,p-q+1}+(n+q+1)P_{q-2}^{n,p-q+1}-(n+p+q+1)P_{q-1}^{n,p-q}\\
				&=(n+p+q+1)P_q^{n-1,p-q}.
			\end{align*}
			where we also used relation (\ref{jacobi1}) twice. The last equality is true due to relation (\ref{jacobi3}). For the $\uu$-terms we therefore get
			\begin{align}
				\uu A^{p-q}C^qD^q&\Big(G_1(s)+\beta G_2(s)\Big)=\uu A^{p-q}C^qD^q(n+p+q+1)(p-1+\beta)P_q^{n-1,p-q}\nonumber\\
				\label{uterms}&=(p+\beta)\uu A^{p-q}C^qD^q(n+p+q+1)P_q^{n-1,p-q},
			\end{align}
			where we used the commutator relation (\ref{comm1}) in Lemma \ref{beta}.\\
			\indent Combining the $\uz\uzd\uu$-parts of terms (\ref{term1}) and (\ref{term4}) gives
			\begin{align}
				&\uz\uzd\uu A^{p-q}C^{q-1}D^qp\big(2P_{q-1}^{n,p-q}-(p-1)P_{q-2}^{n+1,p-q}+(n+p)P_{q-1}^{n+1,p-q}\big)\nonumber\\
				=&\uz\uzd\uu A^{p-q}C^{q-1}D^qp(2P_{q-1}^{n,p-q}+(n+p+q-1)P_{q-1}^{n,p-q})\nonumber\\
				\label{zzuterms}=&\uz\uzd\uu A^{p-q}C^{q-1}D^qp(n+p+q+1)P_{q-1}^{n,p-q}.
			\end{align}
			where we used relation (\ref{jacobi3}) of Lemma \ref{JacRec1}.\\
			\indent For the $\uz\uud\uu$ parts of the equations (\ref{term2}) and (\ref{term3}) we have
			\begin{align*}
				-\uz\uud\uu A^{p-q-1}C^qD^qp\big(P_q^{n,p-q-1}-(p-1)P_{q-1}^{n+1,p-q-1}+(n+p)P_q^{n+1,p-q-1}\big).
			\end{align*}
			By adding and subtracting $(n+1+q)P_{q-1}^{n+1,p-q-1}$ and using relation (\ref{jacobi3}) on the last term we get
			\begin{align}
				&-\uz\uud\uu A^{p-q-1}C^qD^qp\big((n+p+q)(P_q^{n+1,p-q-2}-P_{q-1}^{n+1,p-q-1})+P_q^{n,p-q-1}\big)\nonumber\\
				\label{zuuterms}=&-\uz\uud\uu A^{p-q-1}C^qD^qp(n+p+q+1)P_q^{n,p-q-1}.
			\end{align}
			Collecting the individual terms (\ref{zterms}), (\ref{uterms}), (\ref{zzuterms}) and (\ref{zuuterms}) completes the proof.
			\end{proof}
			In the final step of the computation we will apply the last Dirac operator $\upu$ from the right on the result of the previous lemma. Quite surprisingly, the end
			result is expressed as a linear combination of six Jacobi polynomials multiplied with suitable Clifford numbers.
			\begin{theorem}
			\label{Dirac4}
				If the operator $\upzd\upz$ is applied from the left and $\upud\upu$ from the right on the harmonic kernel $K_{p+1,q+1}^n$ for $p>q\geq 1$ one gets
				\begin{align*}
					\upzd\upz K_{p+1,q+1}^n\upud\upu=&c_{p+1,q+1}(p+1)(n+p+q+1)\la z,u\ra ^{p-q-1}\la z,z\ra ^{q-1}\la u,u\ra ^{q-1}\\
					\times&\bigg(\la z,u\ra \la z,z\ra \la u,u\ra (\beta+p)(n-\beta+q)P_{q}^{n-1,p-q}\\
					-&\la z,u\ra \la u,u\ra p(\beta+p)\uz\wedge\uzd P_{q-1}^{n,p-q}-\la z,u\ra \la z,z\ra p(\beta+p)\uu\wedge \uud P_{q-1}^{n,p-q}\\
					-&\la z,u\ra ^2(n+p)(\beta+p)\uzd\uu P_{q-1}^{n,p-q+1}-\la z,z\ra \la u,u\ra p(n-\beta+q)\uz\uud P_{q}^{n,p-q-1}\\
					+&\la z,u\ra p(n+p+q)\uz\uzd\uu\uud P_{q-1}^{n,p-q}\bigg).
				\end{align*}
			\end{theorem}
		\begin{proof}
			We begin again by computing the action of the Dirac operator $\upu$ from the left on the first term of Lemma \ref{Dirac3}, yielding
			\begin{align*}
				\big(A^{p-q}C^qD^q(\beta+p)\uu P_q^{n-1,p-q}\big)\upu&=A^{p-q}C^qD^q(\beta+p)(n-\beta)P_q^{n-1,p-q}\\
				&+A^{p-q}C^qD^{q-1}(\beta+p)\uu\uud qP_q^{n-1,p-q}\\
				&+A^{p-q}C^qD^q(\beta+p)\uu\big(\uzd\frac{s}{B}-\uud\frac{s}{D}\big)2\big(P_q^{n-1,p-q}\big)'.
			\end{align*}
			By collecting with respect to scalars and bivectors and using $\uu\uzd\frac{s}{B}=(B-\uzd\uu)\frac{s}{B}=s-\uzd\uu\frac{A}{CD}$ we get
			\begin{align*}
				\big(A^{p-q}C^qD^q(\beta+p)\uu P_q^{n-1,p-q}\big)\upu&=A^{p-q}C^qD^q(\beta+p)\big((n-\beta)P_q^{n-1,p-q}+2s\big(P_q^{n-1,p-q}\big)'\big)\\
				&+A^{p-q}C^qD^{q-1}(\beta+p)\uu\uud\big(qP_q^{n-1,p-q}-2s\big(P_q^{n-1,p-q}\big)'\big)\\
				&-A^{p-q+1}C^{q-1}D^{q-1}(\beta+p)\uzd\uu(n+p)P_{q-1}^{n,p-q+1},
			\end{align*}
			where we also used the derivative relation (\ref{jacobi5}) of Lemma \ref{JacRec1} on the last term. Here we can use relation (\ref{jacobi6}) 
			on the second term and, after adding and subtracting $qP_q^{n-1,p-q}$, also on the first one, yielding
			\begin{align*}
				\big(A^{p-q}C^qD^q(\beta+p)\uu P_q^{n-1,p-q}\big)\upu&=A^{p-q}C^qD^q(\beta+p)\big((n-\beta+q)P_q^{n-1,p-q}-(qP_q^{n-1,p-q}-2s\big(P_q^{n-1,p-q}\big)')\big)\\
				&+A^{p-q}C^qD^{q-1}(\beta+p)\uu\uud\big(qP_q^{n-1,p-q}-2s\big(P_q^{n-1,p-q}\big)'\big)\\
				&-A^{p-q+1}C^{q-1}D^{q-1}(\beta+p)\uzd\uu(n+p)P_{q-1}^{n,p-q+1}\\
				&=A^{p-q}C^qD^q(\beta+p)(n-\beta+q)P_q^{n-1,p-q}\\
				&+A^{p-q}C^qD^q(\beta+p)pP_{q-1}^{n,p-q}\\
				&-A^{p-q}C^qD^{q-1}(\beta+p)p\uu\uud P_{q-1}^{n,p-q}\\
				&-A^{p-q+1}C^{q-1}D^{q-1}(\beta+p)\uzd\uu(n+p)P_{q-1}^{n,p-q+1}.
			\end{align*}
		By writing the Clifford product $\uu\uud=\uu\wedge\uud+\frac{1}{2}D$ as a sum of a wedge and a dot product we get 
		\begin{align}
				\label{term6}\big(A^{p-q}C^qD^q(\beta+p)\uu P_q^{n-1,p-q}\big)\upu&=A^{p-q}C^qD^q(\beta+p)(n-\beta+q)P_q^{n-1,p-q}\\
				&+\frac{1}{2}A^{p-q}C^qD^q(\beta+p)pP_{q-1}^{n,p-q}\nonumber\\
				&-A^{p-q}C^qD^{q-1}(\beta+p)p\uu\wedge\uud P_{q-1}^{n,p-q}\nonumber\\
				&-A^{p-q+1}C^{q-1}D^{q-1}(\beta+p)\uzd\uu(n+p)P_{q-1}^{n,p-q+1}.\nonumber
			\end{align}
		Applying $\upu$ from the left on the second term of Lemma \ref{Dirac3} results in
		\begin{align*}
			\big(-A^{p-q-1}C^qD^q\uz\uud\uu pP_q^{n,p-q-1}\big)\upu&=-A^{p-q-1}C^qD^q\uz\uud(n-\beta)pP_q^{n,p-q-1}\\
			&-A^{p-q-1}C^qD^{q-1}\uz\uud\uu\uud pqP_q^{n,p-q-1}\\
			&-A^{p-q-1}C^qD^q\uz\uud\uu\big(\uzd\frac{s}{B}-\uud\frac{s}{D}\big)2p\big(P_q^{n,p-q-1}\big)'.
		\end{align*}
		We use $\uz\uud\uu\uzd=\uz\uud B+\uz\uzd D-\uz\uzd\uu\uud$ and $\uz\uud\uu\uud=\uz\uud D$ to get
		\begin{align*}
		\big(-A^{p-q-1}C^qD^q\uz\uud\uu pP_q^{n,p-q-1}\big)\upu&=-A^{p-q-1}C^qD^q\uz\uud(n-\beta+q)pP_q^{n,p-q-1}\\
			&+A^{p-q-1}C^qD^q\uz\uud2sp\big(P_q^{n,p-q-1}\big)'-A^{p-q-1}C^qD^q\uz\uud2sp\big(P_q^{n,p-q-1}\big)'\\
			&-A^{p-q}C^{q-1}D^q\uz\uzd p(n+p)P_{q-1}^{n+1,p-q}\\
			&+A^{p-q}C^{q-1}D^{q-1}\uz\uzd\uu\uud p(n+p)P_{q-1}^{n+1,p-q},
		\end{align*}
		where we also substituted $2\big(P_q^{n,p-q-1}\big)'=(n+p)P_{q-1}^{n+1,p-q}$.	Expanding the product $\uz\uzd=\uz\wedge\uzd+\frac{1}{2}C$
		gives
		\begin{align}
		\label{term7}\big(-A^{p-q-1}C^qD^q\uz\uud\uu pP_q^{n,p-q-1}\big)\upu&=-A^{p-q-1}C^qD^q\uz\uud(n-\beta+q)pP_q^{n,p-q-1}\\
			&-A^{p-q}C^{q-1}D^q\uz\wedge\uzd p(n+p)P_{q-1}^{n+1,p-q}\nonumber\\
			&-\frac{1}{2}A^{p-q}C^{q}D^qp(n+p)P_{q-1}^{n+1,p-q}\nonumber\\
			&+A^{p-q}C^{q-1}D^{q-1}\uz\uzd\uu\uud p(n+p)P_{q-1}^{n+1,p-q}.\nonumber
		\end{align}
		For the third term of Lemma \ref{Dirac3} we get
		\begin{align*}
			\big(A^{p-q}C^{q-1}D^q\uz\uzd\uu pP_{q-1}^{n,p-q}\big)\upu&=A^{p-q}C^{q-1}D^q\uz\uzd(n-\beta)pP_{q-1}^{n,p-q}\\
			&+A^{p-q}C^{q-1}D^{q-1}\uz\uzd\uu\uud pqP_{q-1}^{n,p-q}\\
			&+A^{p-q}C^{q-1}D^q\uz\uzd\uu\big(\uzd\frac{s}{B}-\uud\frac{s}{D}\big)2p\big(P_{q-1}^{n,p-q}\big)'.
		\end{align*}
		Using $\uz\uzd\uu\uzd=\uz\uzd B$ we collect the result in terms of scalars and bivectors, hence
		\begin{align*}
			\big(A^{p-q}C^{q-1}D^q\uz\uzd\uu pP_{q-1}^{n,p-q}\big)\upu&=A^{p-q}C^{q-1}D^q\uz\uzd(n-\beta)pP_{q-1}^{n,p-q}\\
			&+A^{p-q}C^{q-1}D^q\uz\uzd 2sp\big(P_{q-1}^{n,p-q}\big)'\\
			&+A^{p-q}C^{q-1}D^{q-1}\uz\uzd\uu\uud p\big(qP_{q-1}^{n,p-q}-2s\big(P_{q-1}^{n,p-q}\big)'\big).
		\end{align*}
		By adding und subtracting $-(q-1)P_{q-1}^{n,p-q}$ to the second term and $-P_{q-1}^{n,p-q}$ to the third one we can use relation (\ref{jacobi6}) of Lemma \ref{JacRec2}, hence
		\begin{align*}
			\big(A^{p-q}C^{q-1}D^q\uz\uzd\uu pP_{q-1}^{n,p-q}\big)\upu&=A^{p-q}C^{q-1}D^q\uz\uzd(n-\beta+q-1)pP_{q-1}^{n,p-q}\\
			&-A^{p-q}C^{q-1}D^q\uz\uzd p\big((q-1)P_{q-1}^{n,p-q}-2s\big(P_{q-1}^{n,p-q}\big)'\big)\\
			&+A^{p-q}C^{q-1}D^{q-1}\uz\uzd\uu\uud p\big((q-1)P_{q-1}^{n,p-q}-2s\big(P_{q-1}^{n,p-q}\big)'\big)\\
			&+A^{p-q}C^{q-1}D^{q-1}\uz\uzd\uu\uud pP_{q-1}^{n,p-q}\\
			&=A^{p-q}C^{q-1}D^q\uz\uzd(n-\beta+q-1)pP_{q-1}^{n,p-q}\\
			&+A^{p-q}C^{q-1}D^q\uz\uzd p(p-1)P_{q-2}^{n+1,p-q}\\
			&-A^{p-q}C^{q-1}D^{q-1}\uz\uzd\uu\uud p(p-1)P_{q-2}^{n+1,p-q}\\
			&+A^{p-q}C^{q-1}D^{q-1}\uz\uzd\uu\uud pP_{q-1}^{n,p-q}.
		\end{align*}
		Using $\uz\uzd=\uz\wedge\uzd+\frac{1}{2}C$ in the first and second term results in
		\begin{align}
			\label{term8}\big(A^{p-q}C^{q-1}D^q\uz\uzd\uu pP_{q-1}^{n,p-q}\big)\upu&=\frac{1}{2}A^{p-q}C^{q}D^q(n-\beta+q-1)pP_{q-1}^{n,p-q}\\
			&+\frac{1}{2}A^{p-q}C^{q}D^qp(p-1)P_{q-2}^{n+1,p-q}\nonumber\\
			&+A^{p-q}C^{q-1}D^q\uz\wedge\uzd(n-\beta+q-1)pP_{q-1}^{n,p-q}\nonumber\\
			&+A^{p-q}C^{q-1}D^q\uz\wedge\uzd p(p-1)P_{q-2}^{n+1,p-q}\nonumber\\
			&-A^{p-q}C^{q-1}D^{q-1}\uz\uzd\uu\uud p(p-1)P_{q-2}^{n+1,p-q}\nonumber\\
			&+A^{p-q}C^{q-1}D^{q-1}\uz\uzd\uu\uud pP_{q-1}^{n,p-q}.\nonumber
		\end{align}
		When collecting those parts of (\ref{term6}), (\ref{term7}) and (\ref{term8}) that only contain scalars or the para-bivector $\beta$ we get
		\begin{align*}
			&A^{p-q}C^qD^q\Big((\beta+p)(n-\beta+q)P_q^{n-1,p-q}\\
			&+\frac{1}{2}p\big((\beta+p)P_{q-1}^{n,p-q}-(n+p)P_{q-1}^{n+1,p-q}\\
			&+(n-\beta+q-1)P_{q-1}^{n,p-q}+(p-1)P_{q-2}^{n+1,p-q}\big)\Big)\\
			=&A^{p-q}C^qD^q\Big((\beta+p)(n-\beta+q)P_q^{n-1,p-q}\\
			&+\frac{1}{2}p\big((n+p+q-1)P_{q-1}^{n,p-q}-(n+p)P_{q-1}^{n+1,p-q}\\
			&+(p-1)P_{q-2}^{n+1,p-q}\big)\Big),
		\end{align*}
		where we can use relation (\ref{jacobi3}) on $(n+p)P_{q-1}^{n+1,p-q}$, hence
		\begin{align*}
			&A^{p-q}C^qD^q\Big((\beta+p)(n-\beta+q)P_q^{n-1,p-q}\\
			&+\frac{1}{2}p\big((n+p+q-1)P_{q-1}^{n,p-q}-(n+p+q-1)P_{q-1}^{n+1,p-q-1}+(n+q)P_{q-2}^{n+1,p-q}\\
			&+(p-1)P_{q-2}^{n+1,p-q}\big)\Big)\\
			=&A^{p-q}C^qD^q\Big((\beta+p)(n-\beta+q)P_q^{n-1,p-q}+\frac{1}{2}p\big((n+p+q-1)(P_{q-1}^{n,p-q}-P_{q-1}^{n+1,p-q-1})\\
			&+(n+p+q-1)P_{q-2}^{n+1,p-q}\big)\Big).
		\end{align*}
		Applying relation (\ref{jacobi1}) on the second term results in
		\begin{align}
			&A^{p-q}C^qD^q\Big((\beta+p)(n-\beta+q)P_q^{n-1,p-q}\nonumber\\
			&+\frac{1}{2}p\big(-(n+p+q-1)P_{q-2}^{n+1,p-q}+(n+p+q-1)P_{q-2}^{n+1,p-q}\big)\Big)\nonumber\\
			\label{betaterm}=&A^{p-q}C^qD^q(\beta+p)(n-\beta+q)P_q^{n-1,p-q}.
		\end{align}
		Collecting all parts of (\ref{term7}) and (\ref{term8}) that contain the bivector $\uz\wedge\uzd$ gives
		\begin{align*}
			&-A^{p-q}C^{q-1}D^q\uz\wedge\uzd p\Big((n+p)P_{q-1}^{n+1,p-q}-(n-\beta+q-1)P_{q-1}^{n,p-q}-(p-1)P_{q-2}^{n+1,p-q}\Big).
		\end{align*}
		As before we subsequently apply relation (\ref{jacobi3}) on $(n+p)P_{q-1}^{n+1,p-q}$ and relation (\ref{jacobi1}) to get
		\begin{align}
			&-A^{p-q}C^{q-1}D^q\uz\wedge\uzd p\Big((n+p+q-1)P_{q-1}^{n+1,p-q-1}-(n+q)P_{q-2}^{n+1,p-q}\nonumber\\
			&-(n-\beta+q-1)P_{q-1}^{n,p-q}-(p-1)P_{q-2}^{n+1,p-q}\Big)\nonumber\\
			=&-A^{p-q}C^{q-1}D^q\uz\wedge\uzd p\Big((n+p+q-1)(P_{q-1}^{n+1,p-q-1}-P_{q-2}^{n+1,p-q})-(n-\beta+q-1)P_{q-1}^{n,p-q}\Big)\nonumber\\
			=&-A^{p-q}C^{q-1}D^q\uz\wedge\uzd p\Big((n+p+q-1)P_{q-1}^{n,p-q}-(n-\beta+q-1)P_{q-1}^{n,p-q}\Big)\nonumber\\
			\label{zzterm}=&-A^{p-q}C^{q-1}D^qp(\beta+p)\uz\wedge\uzd P_{q-1}^{n,p-q}.
		\end{align}
		Collecting parts of equations (\ref{term7}) and (\ref{term8}) that contain $\uz\uzd\uu\uud$ results in
		\begin{align*}
			A^{p-q}C^qD^qp\Big(-(p-1)P_{q-2}^{n+1,p-q}+P_{q-1}^{n,p-q}+(n+p)P_{q-1}^{n+1,p-q}\Big).
		\end{align*}
		Once more we apply recurrence relation (\ref{jacobi3}) on $(n+p)P_{q-1}^{n+1,p-q}$ and relation (\ref{jacobi1}) to get
		\begin{align}
			&A^{p-q}C^qD^qp\Big(-(p-1)P_{q-2}^{n+1,p-q}+P_{q-1}^{n,p-q}+(n+p+q-1)P_{q-1}^{n+1,p-q+1}-(n+q)P_{q-2}^{n+1,p-q}\nonumber\\
			=&A^{p-q}C^qD^qp\Big((p-1)(P_{q-1}^{n+1,p-q+1}-P_{q-2}^{n+1,p-q})+(n+q)(P_{q-1}^{n+1,p-q+1}-P_{q-2}^{n+1,p-q})+P_{q-1}^{n,p-q}\Big)\nonumber\\
			=&A^{p-q}C^qD^qp\Big((p-1)P_{q-1}^{n,p-q}+(n+q)P_{q-1}^{n,p-q}+P_{q-1}^{n,p-q}\Big)\nonumber\\
			\label{zzuuterm}=&A^{p-q}C^qD^qp(n+p+q)P_{q-1}^{n,p-q}.
		\end{align}
		The bivector $\uu\wedge\uud$ as well as the para-bivector $\uzd\uu$ can only be found in equation (\ref{term6}). The para-bivector $\uz\uud$ is 
		part of equation (\ref{term7}). These terms are
		\begin{align}
			\label{uuterm}&-A^{p-q}C^qD^{q-1}(\beta+p)p\uu\wedge\uud P_{q-1}^{n,p-q},\\
			\label{zudterm}&-A^{p-q-1}C^qD^qp(n-\beta+q)\uz\uud P_q^{n,p-q-1}\text{ and}\\
			\label{zuterm}&-A^{p-q+1}C^{q-1}D^{q-1}(\beta+p)(n+p)\uzd\uu P_{q-1}^{n,p-q+1}.
		\end{align}
		Combining the results from (\ref{betaterm}) to (\ref{zuterm}) completes the proof of the theorem.
		\end{proof}
		In the above theorem the condition on the homogeneity $(p+1,q+1)$ of $K_{p+1,q+1}^{n}$ is $p>q\geq 1$. However, the symmetric case of homogeneity $(q+1,p+1)$ can
		immediately be derived from it.
		\begin{corollary}
			\label{symm}
			When considering the harmonic kernel
			\begin{align*}
				K_{q+1,p+1}^n(z,u)=\overline{\la z,u\ra}^{p-q}\la z,z\ra^{q+1}\la u,u\ra^{q+1}P_{q+1}^{n-2,p-q}(2s-1)
			\end{align*}
			of homogeneity $(q+1,p+1)$ with $p>q\geq 1$ and applying the operators $\upzd\upz$ and $\upud\upu$ from the left and right respectively we can use the
			result of Theorem \ref{Dirac4} and the Hermitian symmetry of $K_{p+1,q+1}^n(z,u)=\overline{K_{q+1,p+1}^n(z,u)}=K_{q+1,p+1}^n(u,z)$, that is
			\begin{align*}
				\upzd\upz K_{q+1,p+1}^n(z,u)\upud\upu&=\Big(\upu\upud\overline{K_{q+1,p+1}^n(z,u)}\upz\upzd\Big)^\dagger\\
				&=\overline{\Big(\upud\upu K_{p+1,q+1}^n(u,z)\upzd\upz\Big)^\dagger}.
			\end{align*}
			In the last step the identity $\upzd=-\overline{\upz}$ was used. The same identity for a Clifford vector $\uzd=-\overline\uz$ helps to interpret the above expression
			as the result of Theorem \ref{Dirac4} where the roles of $z$ and $u$ are exchanged and the 4-vector as well as the bivectors are reversed. 
		\end{corollary}
		The result of Theorem \ref{Dirac4} also holds formally for the special case $q=0$ by identifying Jacobi polynomials of negative degree with $0$. That is due to the
		validity of the recurrence formulas in this case as stated in Remark \ref{pzero}.
		\begin{lemma}
			Applying the two operators $\upzd\upz$ and $\upud\upu$ to the harmonic kernel of homogeneity $(p+1,1)$ gives
			\begin{align*}
				\upzd\upz K_{p+1,1}^n\upud\upu&=c_{p+1,1}(p+1)(n+p+1)(n-\beta)\Big(\la z,u\ra^p(\beta+p)-p\la z,u\ra^{p-1}\uz\uud\Big).
			\end{align*}
			For the symmetric case $(1,p+1)$ we get according to Corollary \ref{symm}
			\begin{align*}
				\upzd\upz K_{1,p+1}^n\upud\upu&=c_{1,p+1}(p+1)(n+p+1)\beta\Big(\overline{\la z,u\ra}^p(n-\beta+p)-p\overline{\la z,u\ra}^{p-1}\uzd\uu\Big).
			\end{align*}
		\end{lemma}
		Theorem \ref{Dirac4} does not cover the case where the harmonic kernel $K_{p+1,p+1}^n$ is of homogeneity $(p+1,p+1)$. We give the 
		result in the	following lemma, omitting the proof, as it is similar to the previous case.
		\begin{lemma}
			When applying the operators $\upzd\upz$ and $\upud\upu$ on the harmonic kernel $K_{p+1,p+1}^n$ of homogeneity $(p+1,p+1)$ one gets
			\begin{align*}
				\upzd\upz K_{p+1,p+1}^n\upud\upu&=c_{p+1,p+1}\la z,z\ra^{p-1}\la u,u\ra^{p-1}(p+1)(n+2p+1)\\
				&\times\Big(\la z,z\ra\la u,u\ra(\beta+p)(n-\beta+p)P_p^{n-1,0}-\la u,u\ra\uz\wedge\uzd p(\beta+p)P_{p-1}^{n,0}\\
				&-\la z,z\ra p(\beta+p)\uu\wedge\uud P_{p-1}^{n,0}-\la z,u\ra(n+p)(\beta+p)\uzd\uu P_{p-1}^{n,1}\\
				&-\overline{\la z,u\ra}(n+p)(n-\beta+p)\uz\uud P_{p-1}^{n,1}+p(n+2p)\uz\uzd\uu\uud P_{p-1}^{n,0}\Big).
			\end{align*}
		\end{lemma}
		\indent In the Euclidean case the reproducing kernel of spherical monogenics can be expressed (up to a constant) by two Dirac operators with respect to 
		$\ux$ and $\uy$	acting on the reproducing kernel of (real) spherical harmonics from the left and right respectively. Analogously the reproducing 
		kernel of spherical Hermitian monogenics can be described in terms of two pairs of Dirac operators acting on the reproducing kernel of (complex)
		spherical harmonics which we have just computed.	However, the necessary constant will no longer be scalar as in the Euclidean case but instead a (constant)
		polynomial in the Clifford number $\beta$. We will give this constant in Section \ref{sec:3a}.\\
		Our main theorem is hence
		\begin{theorem}
			\label{RepKernel2}
			For any spherical h-monogenic $M_{s,t}\in\cM_{s,t}^{(j)}$ it holds that
			\begin{align}
				\label{rep3}\la \widetilde{K}_{p,q}^n(\cdot,u),M_{s,t}(\cdot)\ra_{\mS^{2n-1}}&=\delta_{ps}\delta_{qt}M_{s,t}(u),\\
				\label{rep4}\la \widetilde{K}_{p,q}^n(\cdot,u),\uz M_{s,t}(\cdot)\ra_{\mS^{2n-1}}&=0,\\
				\label{rep5}\la \widetilde{K}_{p,q}^n(\cdot,u),\uzd M_{s,t}(\cdot)\ra_{\mS^{2n-1}}&=0,\\
				\label{rep6}\la \widetilde{K}_{p,q}^n(\cdot,u),(c_1\uz\uzd+c_2\uzd\uz)M_{s,t}(\cdot)\ra_{\mS^{2n-1}}&=0,
			\end{align}
			with the reproducing kernel
			\begin{align*}
				\widetilde{K}_{p,q}^n(z,u)=d_{p,q}(\beta)\upzd\upz K_{p+1,q+1}^n(z,u)\upud\upu,
			\end{align*}
			for $p>q\geq 1$ and
			\begin{align*}
				\big((n+p+q+1)^2(n-\beta+q)(\beta+p)\big)d_{p,q}(\beta)=1.\\
			\end{align*}
		\end{theorem}
		\begin{proof}
			For a spherical h-monogenic $M_{s,t}(z)$ of order $(s,t)$ we have
			\begin{align*}
				\la\widetilde{K}_{p,q}^n,M_{s,t}\ra_\partial&=d_{p,q}(\beta)\la\upzd\upz K_{p+1,q+1}^n(z,u)\upud\upu,M_{s,t}\ra_\partial\\
				&=\frac{1}{2}d_{p,q}(\beta)\la\upz K_{p+1,q+1}^n\upud\upu,\uzd M_{s,t}\ra_\partial,
			\end{align*}
			where we used Lemma \ref{FischerSphere2} to push the first Dirac operator $\upzd$ to the second argument.
			Using the definition of the Fischer inner product
			\begin{align*}
				\la\widetilde{K}_{p,q}^n,M_{s,t}\ra_\partial&=\frac{1}{2}d_{p,q}(\beta)\Big[\big(\upz K_{p+1,q+1}^n(\partial,u)\upud\upu\big)^\dagger
				\big(\uzd M_{s,t}(z)\big)\Big]_{z=0}\\
				&=\frac{1}{2}d_{p,q}(\beta)\Big[\big(\upud\upu\overline{K_{p+1,q+1}^n}(\partial,u)\upzd\big)\big(\uzd M_{s,t}(z)\big)\Big]_{z=0}\\
				&=\frac{1}{2}d_{p,q}(\beta)\upud\upu\la\upz K_{p+1,q+1}^n,\uzd M_{s,t}\ra_\partial,
			\end{align*}
			we can pull the Dirac operators with respect to $\uu$ out of the inner product and plug in the computation of $\upz K_{p+1,q+1}^n$ of 
			Lemma \ref{Dirac1}, that is
			\begin{align*}
				\la\widetilde{K}_{p,q}^n,M_{s,t}\ra_\partial&=\frac{1}{2}d_{p,q}(\beta)c_{p+1,q+1}\upud\upu\Big(\Big\la(p+1)\uud A^{p-q-1}C^{q+1}D^{q+1}
				P_{q+1}^{n-1,p-q-1},
				\uzd M_{s,t}{\Big\ra}_\partial\\
				&-\Big\la\uzd A^{p-q}C^{q}D^{q+1}P_{q}^{n-1,p-q},\uzd M_{s,t}{\Big\ra}_\partial\Big).
			\end{align*}
			Note that the first arguments in the inner products resemble the reproducing kernel of spherical harmonics of order $(p,q+1)$ and $(p,q)$. 
			If we use equation (\ref{jacobi6}) of Lemma \ref{JacRec2} we get
			\begin{align*}
				\la\widetilde{K}_{p,q}^n,M_{s,t}\ra_\partial&=\frac{1}{2}d_{p,q}(\beta)c_{p+1,q+1}(p+1)\upud\upu\Big(\Big\la\uud A^{p-q-1}C^{q+1}D^{q+1}
				\frac{1}{{{n-1+p}\choose{p}}}\sum_{j=0}^qc_{p-j,q+1-j}P_{q+1-j}^{n-2,p-q-1},\uzd M_{s,t}{\Big\ra}_\partial\\
				&-\Big\la\uzd A^{p-q}C^{q}D^{q+1}\frac{1}{{{n-1+p}\choose{p}}}\sum_{j=0}^qc_{p-j,q-j}P_{q-j}^{n-2,p-q},\uzd M_{s,t}{\Big\ra}_\partial\Big)\\
				&=\frac{1}{2}d_{p,q}(\beta)\frac{c_{p+1,q+1}(p+1)}{{{n-1+p}\choose{p}}}\upud\upu\Big(\Big\la\uud\sum_{j=0}^q\la z,z\ra^j\la u,u\ra^jK_{p-j,q+1-j}^n,
				\uzd M_{s,t}{\Big\ra}_\partial\\
				&-\Big\la\uzd\sum_{j=0}^q\la z,z\ra^j\la u,u\ra^{j+1}K_{p-j,q-j}^n,\uzd M_{s,t}{\Big\ra}_\partial\Big).
			\end{align*}
			Using the definition of the Fischer inner product we get
			\begin{align*}
			\la\widetilde{K}_{p,q}^n,M_{s,t}\ra_\partial&=\frac{1}{2}d_{p,q}(\beta)(n+p+q+1)\upud\upu
			\Big(\sum_{j=0}^q\Big[\uu 4^j\Delta^j\la u,u\ra^j\overline{K_{p-j,q+1-j}^n}(\partial,u)\big(\uzd M_{s,t}(z)\big)\Big]_{z=0}\\
			&-\sum_{j=0}^q\Big[4^j\Delta^j\la u,u\ra^{j+1}\overline{K_{p-j,q-j}^n}(\partial,u)2\upzd\big(\uzd M_{s,t}(z)\big)\Big]_{z=0}\Big).
			\end{align*}
			Because $\uzd M_{s,t}(z)$ is harmonic for a h-monogenic function $M_{s,t}(z)$ all terms in the sum vanish, except for $j=0$. We therefore have
			\begin{align*}
			\la\widetilde{K}_{p,q}^n,M_{s,t}\ra_\partial&=\frac{1}{2}d_{p,q}(\beta)(n+p+q+1)\upud\upu\Big(
			\uu\Big[\overline{K_{p,q+1}^n}(\partial,u)\big(\uzd M_{s,t}(z)\big)\Big]_{z=0}\\
			&-\Big[\la u,u\ra\overline{K_{p,q}^n}(\partial,u)2\upzd\big(\uzd M_{s,t}(z)\big)\Big]_{z=0}\Big)\\
			&=\frac{1}{2}d_{p,q}(\beta)(n+p+q+1)\upud\upu\Big(\uu\Big\la K_{p,q+1}^n,\uzd M_{s,t}\Big\ra_\partial\\
			&-2\la u,u\ra\Big\la K_{p,q}^n,(n-\beta+t)M_{s,t}(z)\Big\ra_\partial\Big),
			\end{align*}
			where we used the $\mathfrak{sl}(1|2)$ relations, see formulas (\ref{sl12a}) and (\ref{sl12b}). After switching to the spherical inner product the spherical 
			harmonic $\uzd M_{s,t}(z)$ is reproduced by the kernel $K_{p,q+1}^n$ for $s=p$ and $t=q$ in the first term. For the same choice of parameters $s$ and $t$ the 
			spherical harmonic $M_{s,t}(z)$ is reproduced by $K_{p,q}^n$ in the second term, hence
			\begin{align*}
				\la\widetilde{K}_{p,q}^n,M_{s,t}\ra_\partial&=\frac{1}{2}d_{p,q}(\beta)(n+p+q+1)\upud\upu\Big(2^{p+q+1}\big(n)_{p+q+1}\uu\Big\la K_{p,q+1}^n,
				\uzd M_{s,t}\Big\ra_{\mS^{2n-1}}\\
				&-2^{p+q}\big(n)_{p+q}\la u,u\ra\Big\la K_{p,q}^n,2(n-\beta+t)M_{s,t}(z)\Big\ra_{\mS^{2n-1}}\Big)\\
				&=2^{p+q}\big(n)_{p+q}d_{p,q}(\beta)(n+p+q+1)\delta_{sp}\delta_{tq}\\
				&\times\Big((n+p+q)\upud\upu\big(\uu\uud M_{s,t}(u)\big)-(n-\beta+q)\upud\upu\big(\la u,u\ra M_{s,t}(u)\big)\Big)\\
				&=2^{p+q}\big(n)_{p+q}d_{p,q}(\beta)(n+p+q+1)\delta_{sp}\delta_{tq}\\
				&\times\Big((n+p+q)\upud\big((\beta+p)\uud M_{s,t}(u)\big)-(n-\beta+q)\upud\big(\uud M_{s,t}(u)\big)\Big)\\
				&=2^{p+q}\big(n)_{p+q}d_{p,q}(\beta)(n+p+q+1)\delta_{sp}\delta_{tq}\\
				&\times\Big((n+p+q)(\beta+p+1)(n-\beta+q)M_{s,t}(u)\big)-(n-\beta+q)(n-\beta+q)M_{s,t}(u)\big)\Big)\\
				&=2^{p+q}\big(n)_{p+q}d_{p,q}(\beta)(n+p+q+1)^2(n-\beta+q)(\beta+p)\delta_{sp}\delta_{tq}M_{s,t}(u),
			\end{align*}
			where we again used formulas (\ref{sl12a}) and (\ref{sl12b}) with respect to $\uu$ and $\uud$. In terms of the spherical inner product we therefore have
			\begin{align*}
				\la\widetilde{K}_{p,q}^n,M_{s,t}\ra_{\mS^{2n-1}}&=\frac{1}{2^{p+q}\big(n)_{p+q}}\la\widetilde{K}_{p,q}^n,M_{s,t}\ra_\partial\\
				&=d_{p,q}(\beta)(n+p+q+1)^2(n-\beta+q)(\beta+p)\delta_{sp}\delta_{tq}M_{s,t}(u)=\delta_{sp}\delta_{tq}M_{s,t}(u),
			\end{align*}
			which completes the proof of (\ref{rep3}). For the statements (\ref{rep4}) to (\ref{rep6}) we again consider the Fischer inner product
			\begin{align*}
				\la \widetilde{K}_{p,q}^n(\cdot,u),\uz M_{s,t}(\cdot)\ra_{\partial}&=\la \upzd\upz K_{p+1,q+1}^n(\cdot,u)\upud\upu,\uz M_{s,t}(\cdot)\ra_{\partial}\\
				&=2\la \upz\upzd\upz K_{p+1,q+1}^n(\cdot,u)\upud\upu, M_{s,t}(\cdot)\ra_{\partial},
			\end{align*}
			where we used again Lemma \ref{FischerSphere2}. Because of $\upz\upz=0$, as stated in Remark \ref{nil}, we are allowed to add an extra term to the first argument, e.g.
			\begin{align*}
				\la \widetilde{K}_{p,q}^n(\cdot,u),\uz M_{s,t}(\cdot)\ra_{\partial}&=2\la \upz\big(\upzd\upz+\upz\upzd\big)
				K_{p+1,q+1}^n(\cdot,u)\upud\upu, M_{s,t}(\cdot)\ra_{\partial}\\
				&=2\la \upz\Delta K_{p+1,q+1}^n(\cdot,u)\upud\upu, M_{s,t}(\cdot)\ra_{\partial}=0.
			\end{align*}
			Here we also used the fact that $\upzd\upz+\upz\upzd=\Delta$ and the harmonicity of the kernel $K_{p+1,q+1}^n(z,u)$. Equation (\ref{rep4}) follows then from the
			proportionality of the Fischer and the spherical inner product. The proofs of statements (\ref{rep5}) and (\ref{rep6}) are analogous and thereby omitted. 
		\end{proof}
		\begin{remark}
			Although a prerequisite of Theorem \ref{RepKernel2} is that $p>q\geq 1$, the statement remains true for all choices of homogeneity $(p,q)$. The proof of this is
			similar to the one shown in the above theorem.
		\end{remark}
		\subsection{Normalization}
			\label{sec:3a}
			In the case of spherical monogenics of order $k$ the reproducing kernel can be obtained up to a constant by letting two Dirac operators act on the
			harmonic kernel of order $(k+1)$. As seen in Theorem \ref{RepKernel1} this normalization constant is found to be $c_k=-(m+2k)^{-2}$. For the
			reproducing kernel of spherical h-monogenics the necessary normalization constant will not be scalar, which follows from the condition we
			derived in Theorem \ref{RepKernel2}, i.e.
			\begin{align}
				\label{norm}(n+p+q+1)^2(n-\beta+q)(\beta+p)d_{p,q}(\beta)=1.
			\end{align}
			This constant $d_{p,q}(\beta)$ has to invert a quadratic polynomial in $\beta$ and therefore also has to be a polynomial
			in $\beta$. For more information on these so-called	spin-Euler polynomials we refer to \cite{fundaments,fundaments2} and \cite{stirling}. In regard of the
			factorial property (\ref{betafactor})	in Lemma \ref{beta} a suitable basis for these polynomials seems to be given by the Lagrange polynomials.
			Indeed when considering the Lagrange polynomials for the points $x_j=j$ for $j=0,\cdots,n$
			\begin{align*}
				L_k(x)=\prod_{\substack{j=0\\j\neq k}}^n\frac{x-j}{k-j}
			\end{align*}
			and evaluating them in $\beta$ we have some useful properties, given in the following lemma.
			\begin{lemma}
				\label{lagrange}
				For the Lagrange polynomials $L_j$ in $\beta$ it holds that
				\begin{align}
					\label{L1}\sum_{j=0}^nL_j(\beta)&=1,\\
					\label{L2}\beta L_j(\beta)&=jL_j(\beta)\hspace{10mm}j=0,\cdots,n,\\
					\label{L3}L_j(n-\beta)&=L_{n-j}(\beta)\hspace{8mm}j=0,\cdots,n.
				\end{align}
			\end{lemma}
			\begin{proof}
				Property (\ref{L1}) holds for Lagrange polynomials in general and thereby also for $L_j(\beta)$. Property (\ref{L2}) follows from the factorial 
				property (\ref{betafactor}), that is
				\begin{align*}
					\beta L_j(\beta)-jL_j(\beta)&=(\beta-j)L_j(\beta)=(\beta-j)\prod_{\substack{l=0\\l\neq j}}^n\frac{\beta-l}{j-l}
					=\frac{\beta(\beta-1)\cdots(\beta-n)}{\prod_{\substack{l=0\\l\neq j}}^n(j-l)}=0.
				\end{align*}
				For the symmetry property (\ref{L3}) we have
				\begin{align*}
					L_j(n-\beta)&=\prod_{\substack{l=0\\l\neq j}}^n\frac{(n-\beta)-l}{j-l}=\prod_{\substack{l=0\\l\neq j}}^n\frac{-\beta+(n-l)}{-(n-j)+(n-l)}
					=\frac{(-1)^n}{(-1)^n}\prod_{\substack{l=0\\l\neq j}}^n\frac{\beta-(n-l)}{(n-j)-(n-l)}\\
					&=\prod_{\substack{k=0\\k\neq n-j}}^n\frac{\beta-k}{(n-j)-k}=L_{n-j}(\beta).
				\end{align*}
			\end{proof}
		With these tools at hand we can solve equation (\ref{norm}) to determine the normalization constant $d_{p,q}(\beta)$. The result of this computation is
		given in the following lemma.
			\begin{lemma}
				The normalization constant $d_{p,q}(\beta)$ of the reproducing kernel of spherical	h-monogenics is given by
				\begin{align*}
					d_{p,q}(\beta)=\sum_{j=0}^nd_{p,q}^{(j)}L_j(\beta),
				\end{align*}
				where $L_j(\beta)=\prod_{\substack{l=0\\l\neq j}}^n\frac{\beta-l}{j-l}$ are the Lagrange polynomials and $d_{p,q}^{(j)}=(n+p+q+1)^{-2}(n-j+q)^{-1}(j+p)^{-1}$.
			\end{lemma}
			\begin{proof}
				In order to show that the given constant solves equation (\ref{norm}) we write $d_{p,q}(\beta)$ in terms of the Lagrange polynomials $L_j(\beta)$, hence
				\begin{align*}
					d_{p,q}(\beta)=\sum_{j=0}^nd_{p,q}^{(j)}L_j(\beta).
				\end{align*}
				For the normalization condition (\ref{norm}) we get
				\begin{align*}
					(n+p+q+1)^2(n-\beta+q)(\beta+p)d_{p,q}(\beta)&=(n+p+q+1)^2(n-\beta+q)(\beta+p)\sum_{j=0}^nd_{p,q}^{(j)}L_j(\beta)\\
					&=\sum_{j=0}^nd_{p,q}^{(j)}(n+p+q+1)^2(n-\beta+q)(\beta+p)L_j(\beta)\\
					&=\sum_{j=0}^nd_{p,q}^{(j)}(n+p+q+1)^2(n-j+q)(j+p)L_j(\beta),
				\end{align*}
				where we used property (\ref{L2}) of Lemma \ref{lagrange} twice in a row in the last step. If we choose the coefficients 
				$d_{p,q}^{(j)}=(n+p+q+1)^{-2}(n-j+q)^{-1}(j+p)^{-1}$ we get
				\begin{align*}
					(n+p+q+1)^2(n-\beta+q)(\beta+p)d_{p,q}(\beta)=\sum_{j=0}^nL_j(\beta)=1,
				\end{align*}
				where the last equality is true due to property (\ref{L1}) of Lemma \ref{lagrange}.
			\end{proof}

\newpage
\begin{landscape}
		\begin{table}[h]
		\begin{tabular}{|c|c|c|c|c|}
			\hline
			&Harmonic analysis&Complex harmonic analysis&Euclidean Clifford analysis&Hermitian Clifford analysis\\
			\hline\hline
			space/values&$\mR^m\rightarrow\mC$&$\mR^{2n}\rightarrow\mC$&$\mR^m\rightarrow\mathcal{C\ell}_m$&$\mR^{2n}\rightarrow\mC_{2n}$\\
			\hline
			&$\Delta=\sum\limits_{j=1}^m\partial_{x_j}\partial_{x_j}$&
			$\Delta=\sum\limits_{j=1}^n\partial_{z_j}\dbar_{z_j}$&
			$\upx=\sum\limits_{j=1}^me_j\partial_{x_j}$&
			$\upz=\sum\limits_{j=1}^nf_j\dbar_{z_j}\hspace{0.6cm}\upzd=\sum\limits_{j=1}^nf_j^\dagger\partial_{z_j}$\\
			operators&$\mE=\sum\limits_{j=1}^mx_j\partial_{x_j}$&
			$\mE_z=\sum\limits_{j=1}^nz_j\partial_{z_j}\hspace{0.4cm}\mE_{\zc}=\sum\limits_{j=1}^n\zc_j\dbar_{z_j}$&
			$\mE=\sum\limits_{j=1}^mx_j\partial_{x_j}$&
			$\mE_z=\sum\limits_{j=1}^nz_j\partial_{z_j}\hspace{0.6cm}\mE_{\zc}=\sum\limits_{j=1}^n\zc_j\dbar_{z_j}$\\
			&$r^2=\la x,x\ra$&
			$r^2=\la z,z\ra$&
			$r^2=-\ux^2$&
			$r^2=\{\uz,\uzd\}$\\
			\hline
			dual pair&$\mathfrak{sl}_2\times O(m)$&$\mathfrak{gl}_2\times U(n)$&$\mathfrak{osp}(1|2)\times Spin(m)$&$\mathfrak{sl}(1|2)\times U(n)$\\
			\hline
			null solutions&$\cH_k=\cP_k\cap\ker\Delta$&$\cH_{p,q}=\cP_{p,q}\cap\ker\Delta$&$\cM_k=\cP_k\cap\ker\upx$&$\cM_{p,q}^{(j)}=\cP_{p,q}^{(j)}\cap\ker\upz\cap\ker\upzd$\\
			\hline
			reproducing kernel&$K_k^m$&$K_{p,q}^{n}$&$\widetilde{K}_k^m=c_k\upx K_k^m\upy$&$\widetilde{K}_{p,q}^{n}=d_{p,q}(\beta)\upzd\upz K_{p,q}^{n}\upud\upu$\\
			\hline
			theorem&Theorem \ref{harmkernel1}&Theorem \ref{harmkernel2}&Theorem \ref{RepKernel1}&Theorem \ref{RepKernel2}\\
			\hline
			Fischer decomposition&
			$\cP_k^m=\bigoplus\limits_{j=0}^{\lfloor\frac{k}{2}\rfloor}|x|^{2j}\cH_{k-2j}^m$&
			$\cP_{p,q}^{2n}=\bigoplus\limits_{j=0}^{\min(p,q)}|z|^{2j}\cH_{p-j,q-j}^{2n}$&
			$\cP_k^m=\bigoplus\limits_{j=0}^k\ux^j\cM_{k-j}^m$&
			see \cite{fischer}, page 310\\
			\hline
		\end{tabular}
	\caption{summary of the results}
	\label{table}
	\end{table}
	\end{landscape}

\bibliographystyle{plain}

\begin{thebibliography}{23}

\bibitem{AAR}
	{G. Andrews, R. Askey and R. Roy},
	\emph{Special functions},
	{Cambridge University Press},
	{Cambridge},
	{1999.}

\bibitem{askey}
	{R. Askey},
	\emph{Orthogonal Polynomials and Special Functions},
	{SIAM},
	{1975.}

\bibitem{stras}
	{A. Bezubik and A. Strasburger},
	\emph{On spherical expansions of smooth $SU(n)$-zonal functions on the unit sphere in $\mathbb{C}^n$},
	{J. Math. Anal. Appl.},
	{404},
	{570--578},
	{2013.}

\bibitem{fundaments}
	{F. Brackx, J. Bure\v{s}, H. De Schepper, D. Eelbode, F. Sommen and V. Sou\v{c}ek},
	\emph{Fundaments of {H}ermitean {C}lifford Analysis Part {I}: Complex Structure},
	{Complex Anal. Oper. Theory},
	{1},
	{341--365},
	{2007.}
	
\bibitem{fundaments2}
	{F. Brackx, J. Bure\v{s}, H. De Schepper, D. Eelbode, F. Sommen and V. Sou\v{c}ek},
	\emph{Fundaments of {H}ermitean {C}lifford analysis part {II}: splitting of h-monogenic equations},
	{Complex Var. Elliptic Equ.},
	{52},
	{1063--1079},
	{2007.}

\bibitem{red}
	{F. Brackx, R. Delanghe and F. Sommen},
	\emph{Clifford Analysis},
	{Pitman (Advanced Publishing Program)},
	{Boston},
	{1982.}

\bibitem{howe}
	{F. Brackx, H. De Schepper, D. Eelbode, V. Sou\v{c}ek},
	\emph{The {H}owe dual pair in {H}ermitean {C}lifford analysis},
	{Rev. Mat. Iberoamericana},
	{26},
	{2},
	{449--479},
	{2010.}


	
	
	

\bibitem{fischer}
	{F. Brackx, H. De Schepper and V. Sou\v{c}ek},
	\emph{Fischer decompositions in {E}uclidean and {H}ermitean {C}lifford analysis},
	{Arch. Math. (Brno)},
	{46},
	{301--321},
	{2010.}

\bibitem{xu}
	F. Dai and Y. Xu,
	\emph{Approximation Theory and Harmonic Analysis on Spheres and Balls},
	Springer,
	London,
	2013.


\bibitem{deform}
	{H. De Bie, B. {\O}rsted, P. Somberg and V. Sou\v{c}ek},
	\emph{The {C}lifford deformation of the {H}ermite semigroup},
	{SIGMA},
	{9},
	{2013.}

\bibitem{green}
	{R. Delanghe, F. Sommen and V. Sou\v{c}ek},
	\emph{Clifford algebra and spinor-valued functions},
	{Kluwer Academic Publishers Group},
	{Dordrecht},
	{1992.}

\bibitem{hmonogenics}
	{D. Eelbode},
	\emph{Irreducible $\mathfrak{sl}(m)$-modules of {H}ermitean monogenics},
	{Complex Var. Elliptic Equ.},
	{53},
	{10},
	{975--987},
	{2008.}

\bibitem{stirling}
	{D. Eelbode},
	\emph{Stirling Numbers and Spin-{E}uler Polynomials},
	{Experiment. Math.},
	{16},
	{55--66},
	{2007.}
	
\bibitem{Fi}
	{E. Fischer},
	\emph{{\"U}ber die {D}ifferentiationsprozesse der {A}lgebra},
	{J. f\"ur Math.},
	{148},
	{1--78},
	{1917.}
	
\bibitem{koorn}
	{T.H. Koornwinder},
	\emph{The Addition Formula for {J}acobi Polynomials and Spherical Harmonics},
	{SIAM J. Appl. Math.},
	{25},
	{236--246},
	{1973.}

\bibitem{S}
{H. Shapiro},
\emph{An algebraic theorem of {E}. {F}ischer, and the holomorphic {G}oursat problem},
{Bull. London Math. Soc.},
{21},
{6},
{513--537},
{1989.}

\bibitem{stein}
	{E. Stein and G. Weiss},
	\emph{Introduction to Fourier Analysis on Euclidean Spaces},
	{Princeton University Press},
	{Princeton},
	{1971.}

\bibitem{gabor}
	{G. Szeg\H{o}},
	\emph{Orthogonal Polynomials},
	{American Mathematical Society},
	{New York},
	{1939.}

\end{thebibliography}

\end{document}